\documentclass[12pt,oneside]{amsart}

\usepackage{amssymb}
\usepackage{amscd}

\theoremstyle{plain}
\newtheorem{theorem}{Theorem}[section]
\newtheorem{lemma}[theorem]{Lemma}
\newtheorem{corollary}[theorem]{Corollary}
\newtheorem{proposition}[theorem]{Proposition}

\theoremstyle{definition}
\newtheorem{definition}[theorem]{Definition}
\newtheorem{example}[theorem]{Example}
      
\theoremstyle{remark}

\newcommand{\C}{C_{\phi}}
\newcommand{\Cn}{\mathbb{C}^n}
\newcommand{\CN}{\mathbb{C}^N}

\newcommand{\Bn}{\mathbb{B}^n}

\newcommand{\D}{\mathcal{D}}
\newcommand{\A}{A_H^2(\D)}
\newcommand{\h}{\mathcal{H}}
\newcommand{\p}{\phi'(0)^t}

\newcommand{\E}{\epsilon}

      \makeatletter
      \def\@setcopyright{}
      \def\serieslogo@{}
      \makeatother
    
\begin{document}


   \author{Robert Bridges}
   \address{Purdue University, West Lafayette, IN}
   \email{bridges@purdue.edu}


     

   \title[Schroeder's Functional Equation]{A Solution to Schroeder's Equation in Several Variables}
   
   \begin{abstract}
    Let $\phi$ be a self-map of $\Bn$, the unit ball in $\Cn$, fixing $0$, and having full-rank at $0$.  If $\phi'(0)\neq0$, Koenigs proved in 1884 that in the well-known case $n=1$, Schroeder's equation, $ f\circ \phi = \lambda f$ has a solution $f$, which is bijective near $0$ precisely when $\lambda=\phi'(0)$.  In 2003, Cowen and MacCluer formulated the analogous problem in $\Cn$ (for a non-negative integer $n$) by defining Schroeder's equation in several variables as $F\circ \phi = \phi'(0) F$ and giving appropriate assumptions on $\phi$.  The 2003 Cowen and MacCluer paper also provides necessary and sufficient conditions for an analytic solution, $F$ taking values in $\Cn$ and having full-rank near $0$ under the additional assumption that $\phi'(0)$ is diagonalizable.  The main result of this paper gives necessary and sufficient conditions for a Schroeder solution $F$ which has full rank near 0 without the added assumption of diagonalizability. More generally, it is proven in this paper that the functional equation $\C F=\phi'(0)^kF$ with $k$ a positive integer, is always solvable with an $F$ whose component functions are linearly independent, but if $k>1$ any such $F$ cannot be injective near 0.  In 2007 Enoch provides many theorems giving formal power series solutions to Schroeder's equation in several variables.  It is also proved in this note that any formal power series solution indeed represents an analytic function on the ball.     
   \end{abstract}



   \thanks{This research was partially supported by NSF Analysis and Cyber-enabled Discovery Initiative Programs,
grant number DMS-1001701}
\thanks{Section \ref{Section: generalized solutions} was inspired by a short conversation with \^{Z}eljko \^{C}u\^{c}kovi\'{c} that took place at the 27th Southeast Analysis Meetings, University of Florida Gainesville, March 17-19, 2011.  The author thanks Professor \^{C}u\^{c}kovi\'{c}.}
   \thanks{The author thanks Carl C. Cowen, his thesis advisor the the advice, encouragement, and guidance.}


   \date{\today}
   
    \maketitle
    
\section{Introduction}
Let $\phi$ be a self-map of $\mathbb{D}=\{z\in \mathbb{C}: |z|<1\}$, fixing $0$, and not a disk automorphism.  To avoid trivialities, assume that $\phi$ is not the zero function.  Schroeder's equation is $f\circ \phi= \lambda f$, where $f$ is an unknown analytic function on $\D$, and $\lambda$ an unknown complex number.  If $C_{\phi}$ denotes the composition operator, which sends a function $f$ defined on $\mathbb{D}$ to $f\circ\phi$, then Schroeder's equation is the eigenvalue equation, $\C f= \lambda f$.   Koenigs showed in 1884 that Schroeder's equation has a solution, $f,$ which is also bijective near 0, if and only if $\lambda=\phi'(0)\neq 0$ \cite{Koen}.  Now if $f'(0)\neq 0$, we can view Schroeder's equation not only as an eigenvalue equation, but as a change of variables near 0, the attracting fixed point of $\phi$.  Explicitly,  $\phi(z)=f^{-1}(\phi'(0)z)$ near 0.  Such a solution was one of the first steps in understanding intertwining maps and models of iteration, a theory which has been foundational for the understanding of composition operators in one variable (for example, see \cite[Section 2.4]{CMbook}).  
  
  In hopes of generalization,  Cowen and MacCluer have formulated an analogous problem in $\Cn$ (for arbitrary $n$), namely to solve the functional equation 
  \begin{equation}
\label{intro equation}  
  \C F=\phi'(0)F
  \end{equation}
   with an analytic $F:\Bn \to \Cn$, where $\phi$ is a given analytic self-map of $\Bn$, the unit ball in $\Cn$ \cite{CM03}.   The hypotheses taken on $\phi$ are that $\phi(0)=0$, $\phi'(0)$ has full rank, and $\phi$ is not unitary on any slice.  Notice that the statement ``$|\phi(z)|<|z|$ for any $0<z<1$" is equivalent to ``$\phi $ is not unitary on a slice" via Schwarz's lemma.  Since $\phi$ is $\Cn$-valued, we see that $\phi'(0)$ is an $n\times n$ matrix, so $F$ must be a column vector with $n$ components, that is, 
      $$F=\left[\begin{matrix}f_1 \\ \vdots \\ f_n  \end{matrix} \right]= (f_1, \dots, f_n)^t$$
with each $f_j:\Bn\to \Cn$ analytic so the multiplication on the right side of (\ref{intro equation}) makes sense.  For $n>1$ Equation (\ref{intro equation}) is no longer an eigenvalue equation but is still a change of variables near 0, provided $F$ has full rank near 0.  As in the one variable setting,  $\phi(z)=F^{-1}(\phi'(0)F(z))$ near 0.  Thus, we seek solutions $F:\Bn \to \Cn$ which are analytic and have full rank near 0, and any such $F$ will be referred to as a ``full rank solution."  For more details on the formation of the problem, see \cite{CM03}.  Unlike the single variable case, it is easy to see that if $n>1$ there exist Schroeder solutions which are nontrivial and not full rank either, for example if $\phi(z_1,z_2)=(z_1/2,z_2/4)$ and $F(z_1,z_2)=(z_1, z_1^2)^t.$  
   
   After formulating Schroeder's equation in several variables, \cite{CM03} gives necessary and sufficient conditions for a full rank solution under the additional hypothesis that $\phi'(0)$ is diagonalizable.  Specifically,  their argument exhibits a Bergman space of analytic function on $\Bn$ which is large enough so that $\C$ is compact on it, and then uses the compactness to produce solutions $F=(f_1, ... , f_n)^t$ with each $f_j$ in the Bergman space.  This ensures that  
$F$ is indeed analytic.  The current article adopts the same overall strategy but provides a new viewpoint, namely reducing to the case that $\phi'(0)$ is in Jordan form.  It is this insight that allows us to give necessary and sufficient conditions for a full rank solution under the general hypotheses, that is, without the additional assumption that $\phi'(0)$ is diagonalizable  (see Theorem  \ref{C: main theorem part 2} and \ref{main theorem revisited}).  

It is well-known in the one variable setting that $\phi'(0)^k$ is an eigenvalue of $\C$ for any positive integer $k.$  We consider the analogous equation, 
\begin{equation}
\label{intro general equation}  
  \C F=\phi'(0)^k F
  \end{equation} 
with $k$ any positive integer, 
and show under the same hypotheses on $\phi$ that a solution $F$ with linearly independent component functions always exists (see Theorem \ref{main generalized theorem}). 

\subsection{Outline}
Sections \ref{Sec:SA} and \ref{Sec: N}  will provide the details of reducing to $\phi'(0)$ in Jordan form, introduce necessary notation, and see what is gained by such a simplification.  Reducing to Jordan form is (to the author's knowledge)  not previously considered in the literature, and this new point of view is the ``trick" to understanding solutions.  Section \ref{Sec. Sw/oCRN0} uses the compactness of $\C$ to reduce the problem of finding Schroeder solutions to that of linear algebra, and contains a subsection of the necessary linear algebra results.  It is here that much of the work takes place.  The main result of this section will guarantee the existence of a Schroeder solution, $F$ with linearly independent component functions.  This gives an intermediate solution (between no solution and a full rank solution) when no full rank solution exists.  Section \ref{Sec: R} examines obstructions to a full rank solution, and the role resonant eigenvalues play; specifically, the Jordan form viewpoint allows one to see specifically how a resonant eigenvalue prevents the existence of a solution which is locally univalent near 0. Our main result lies in Section \ref{Sec: CfaFRS}, which tackles the problem of discriminating which $\phi$ admit full rank solutions to Schroeder's equation, and the closing section addresses solutions to (\ref{intro general equation}).

\section{Simplifying Assumptions}
\label{Sec:SA}
   Let $\phi$ will denote an analytic self-map of $\Bn$, fixing 0.  Let $D$ be an invertible $n\times n$ matrix such that $D\phi'(0)D^{-1}$ is in Jordan form.  Notice $D\Bn$ is an ellipsoid, so while $D\phi D^{-1}$ does have a Jordan form derivative at $0$, it is not necessarily defined on $\Bn$.  This $\acute{\mbox{a}}$ priori eliminates using $C_{D\phi D^{-1}}$ on a Hilbert space of analytic functions on $\Bn$.  To continue we will construct a Hilbert space of analytic functions on $D\Bn$, so that $C_{D\phi D^{-1}}$ is not only defined on this space, but is even a compact operator.
 
Recall that given a domain $\mathcal{D}$ of $\Cn$ and $G:\mathcal{D} \to (0,\infty)$,  $A_G^2(\mathcal{D})$ denotes the space of analytic function $f:\mathcal{D}\to \mathbb{C}$ satisfying
$$\|f \|^2:=\int_{\mathcal{D}} |f(z)|^2 G(z)dz <\infty$$
where $dz$ denotes Lebesgue measure on $\mathbb{R}^{2n}.$
When the domain $\mathcal{D}=\Bn$ we may simply write $A_G^2$.  

Clearly $f\mapsto f\circ D^{-1}$ is a bijective map of $\mathcal{O}(\Bn)$ onto $\mathcal{O}(D\Bn)$, where $\mathcal{O}(\mathcal{D})$ denotes the analytic functions on a domain $\mathcal{D}$.  Since, 
$$\int_{\Bn} |f|^2(z) G(z)dz = \int_{D\Bn}|f\circ (D^{-1})|^2(w) \frac{G\circ (D^{-1})(w)}{|\det(D) |^2}dw,$$
we see $A_H^2(D\Bn)$ is isometrically isomorphic to $A_G^2$ for $$H=\frac{G\circ D^{-1}}{|\det(D)|^2}.$$  This gives rise to the following proposition.

\begin{proposition}
\label{SA: p1}
With $\phi$ and $H$ as above, the following are true.
\begin{enumerate}
\item There exists $D$, an invertible $n\times n$ matrix such that $D\phi'(0)D^{-1}$ in Jordan form.
\item $A_G^2$ is isometrically isomorphic to $A_H^2(D\Bn)$ via $\iota(f)=f\circ D^{-1}$.
\item $C_{\phi}$ is bounded (compact) on $A_G^2$ if and only if $C_{D\phi D^{-1}}$ is bounded (compact) on $A_H^2(D\Bn)$. 
\item  $\{z^{\alpha}\}_{\alpha}$ for multi-indices $\alpha \in (\mathbb{N}\cup\{0\})^N$ is a linearly independent subset of both  $A_G^2$ and $A_H^2(D\Bn)$ (although not necessarily orthogonal in $A_H^2(D\Bn)$).
\end{enumerate}
\end{proposition}
In practice we usually represent functions of $A_H^2(D\Bn)$ using their Taylor series.  (4) is stated as a reminder that the monomials are not orthogonal vectors in $A_H^2(D\Bn)$.  Although, with the aid of Proposition \ref{SA: P2} we see that the subsequence of homogeneous polynomials converges in $A_H^2(D\Bn)$, and so this convenient representation makes sense in $A_H^2(D\Bn)$.  
\begin{proof}
All statements up to and including the isometric isomorphism have been proven. 
For (3), noting that the following diagram commutes, 
$$\begin{CD}
A_G^2 @>C_\phi >> A_G^2\\
 @V\iota VV @VV\iota V\\
 A_H^2(D\Bn) @>C_{D\phi D^{-1}}>> A_H^2(D\Bn)
\end{CD}$$
we see $C_{\phi}$ is bounded (compact) on $A_G^2$ if and only if $C_{D\phi D^{-1}}$ is bounded (compact) on $A_H^2(D\Bn)$.

For (4) it is well known that $\{z^{\alpha}\}$ is orthogonal (and therefore linearly independent) in $A_G^2$.  For the last result, assume $$\sum_{j=1}^n c_j z^{\alpha_j} =0$$ in $A_H^2(D\Bn)$ for constants $c_j$.  This means that the polynomial on the left is 0 for every $z$ in the open set $D\Bn$, which yields $c_j=0$ for each $j$. Thus we have shown the linear independence of $\{z^{\alpha}\}$.  

To see orthogonality of $\{z^{\alpha}\}$ can fail in $A_H^2(D\Bn)$, consider $G\equiv 1$ on $\mathbb{B}^2$, and
$$D=\left[ \begin{matrix}
1&1\\
0&1
\end{matrix}\right].$$  
A simple calculation shows
\begin{align*} \langle z_1,z_2 \rangle_{A_H^2(D\mathbb{B}^2)} &=\int_{D\mathbb{B}^2}z_1\bar{z_2}\frac{dz}{|\det(D)|^2}\\
									&=\int_{\mathbb{B}^2}(z_1+z_2)\bar{z_2}dz\\
									&=\langle z_2, z_2 \rangle_{A_G^2} >0. 
\end{align*}
 \end{proof}
\begin{proposition}
\label{SA: P2}
Suppose $f=\sum_{j=0}^{\infty}f_j$ is a member of $A_G^2$ with $f_j$ homogeneous polynomials of degree $j$.  Setting $g=\iota f$ and $g_j=\iota f_j$, we see   
\begin{enumerate}
\item $g_j$ is also a homogeneous polynomial of degree $j$, 
\item in $A_H^2(D\Bn)$ $\langle g_j, g_k \rangle =0 $ for $k\neq j$, and
\item $\sum g_j \to g$ in $A_H^2(D\Bn).$
\end{enumerate}
\end{proposition}
\begin{proof}
For any $\lambda$ in $\Bn$, we have 
\begin{align*}
g_j(\lambda z)& = \iota f_j (\lambda z) \\
			&= f_jD^{-1}(\lambda z) \\			
			&= f_j(\lambda D^{-1}z) \\
			&= \lambda^j f_jD^{-1}(z) = \lambda^j g_n(z)
\end{align*}
which shows (1).  Both (2) and (3) follow immediately from (1) and the fact $\iota$ is an isometric isomorphism. 
 \end{proof}

\begin{lemma}
\label{SA: l1}
With $\phi$ and $D$ as above, $F$ is a Schroeder solution for $C_{D\phi D^{-1}}$ (with domain $D\Bn$) if and only if $D^{-1}FD$ is a Schroeder solution for $\C$ (with domain $\Bn$). 
\end{lemma}
\begin{proof}
Suppose $F$ is a Schroeder solution for $C_{D\phi D^{-1}}$.  Clearly, $D^{-1}FD$ has full rank near $0$ since $F$ does, and $D$ invertible. So we only need to show 
\begin{align*} & \C(D^{-1}FD)=\phi'(0)D^{-1}FD\\
\iff & D^{-1}F(D\phi(z))=\phi'(0)D^{-1}F(Dz) \mbox{ whenever $z$ in $\Bn$}\\
\iff & F(D\phi(z))=D\phi'(0)D^{-1}F(Dz) \mbox{ whenever $z$ in $\Bn$}\\
\iff & F(D\phi(D^{-1}w))=D\phi'(0)D^{-1}F(w) \mbox{ whenever $w$ in $D\Bn$}\\
\iff & C_{D\phi D^{-1}}F=D\phi'(0)D^{-1}F
\end{align*}
so we have the result. 
\end{proof} 

Recall that our goal is to find a Schroeder solution under the hypotheses that $\phi$ is a self-map of $\Bn$ fixing 0, not unitary on a slice, and $\phi'(0)$ invertible.  Cowen and MacCluer provide sufficient conditions on $G$ so that $\C$ is compact on the space $A_G^2$, and we state their theorem below.  As in \cite{CM03}, the compactness of $\C$ will allow us to produce solutions whose component functions lie in $A_G^2$, which ensures the analyticity of $F$.    
\begin{theorem}\cite[Theorem 8]{CM03}
\label{T: CM cpt}
If $\phi:\Bn \to \Bn$ is analytic, fixing 0, and $\phi$ is not unitary on any slice, then $C_{\phi}$ is Hilbert-Schmidt (and therefore compact) on $A_G^2$  for any continuous, non-increasing $G:[0,1)\to (0, \infty)$ such that 
$$\frac{G(r)}{(1-r)^{2n} G(1-\rho(1-r))} $$
is bounded near 1 for any $\rho >1$. In particular, $\C$ is Hilbert-Schmidt on $A_G^2$ for $G(|z|)=\exp(-q/(1-|z|)), q>0.$  
\end{theorem}
Fix $G$ such that $\C$ is compact on the Bergman space $A_G^2$.
\begin{corollary}
$C_{D\phi D^{-1}}$ is compact on $A_H^2(D\Bn)$.
\end{corollary}
Theorem \ref{SA: l1} allows us to show the result for $D\phi D^{-1}.$ Thus, by replacing $\phi$ with $D\phi D^{-1}$ and $A_G^2$ with $A_H^2(D\Bn)$, we may assume $\phi'(0)$ is in Jordan form and retain the compactness of $C_{D\phi D^{-1}}$.  Of course the exchange is not free, and the cost for this simplification is that the domain is now $D\Bn,$ and the convenient basis vectors, namely, $\{z^{\alpha} \}_{\alpha}$ are no longer orthogonal. 

\subsection{Why Jordan Form?}
 The upshot of assuming $\phi'(0)$ is in Jordan form is that we can easily reformulate Schroder's equation into a statement about the chains of $\C$.    

\begin{definition} Let $L$ be a linear operator on a vector space.  We say $e_1, ..., e_k$ are a chain with eigenvalue $\lambda$ and length $k$ if (and only if) they are non-zero vectors satisfying  
\begin{equation}
\label{P: e1}
  (L-\lambda I) e_j=\left\{ \begin{array}{cc}
   e_{j-1} &\mbox{ if } 2\leq j \leq k\\
   0 & \mbox{ if } j=1.
   \end{array}\right.
\end{equation}  
\end{definition} 
  Note that the existence of a chain is equivalent to $[\ker(L-\lambda I)]\cap[ (L-\lambda I)^{k-1}(X)] $ being nontrivial and also that the compression of $L$ to $\langle e_1, .... , e_k\rangle $ is just the upper triangular $\lambda$-Jordan block of length $k$,
  $$\left[ \begin{matrix}
  \lambda 	& 1 	& \\
  			& \ddots & \ddots  \\ 
			&		 & \lambda 	& 1\\
			&		 &			&\lambda 
  \end{matrix}\right].$$
We also notice that reversing the indices of the chain gives $e_k, ... , e_1$, a chain of the lower triangular  $\lambda$-Jordan block
 $$\left[ \begin{matrix}
  \lambda 	&  			& \\
  	1		& \lambda	&   \\ 
			&\ddots	 	& \ddots  \\
			&		 &	1		&\lambda 
  \end{matrix}\right].$$
Consider the following (abstract) example.
\begin{example}
\label{abstract example}
Let
 $$\phi'(0)=
\left[ \begin{matrix}
\lambda & 1\\
0		&\lambda \\
 		&			&\alpha	& 1 		&0\\
 					&&0		& \alpha	&1\\
 					&&0		&0			&\alpha
 \end{matrix} \right]$$ 
 a Jordan form matrix with 2 Jordan blocks, one of length 2 and eigenvalue $\lambda$, and one of length 3 and eigenvalue $\alpha$.  For our problem it will be useful to consider the Jordan blocks as defining chains of $\phi'(0)$.  This  $\phi'(0)$ has two chains, namely 
 $$(\phi'(0)-\lambda I_5): e_2 \mapsto e_1 \mapsto 0,$$
 and
 $$(\phi'(0)-\alpha I_5): e_5 \mapsto e_4 \mapsto e_3 \mapsto 0.$$

Now consider solving Schroeder's equation with this $\phi$.  Writing $F=(f_1, ... , f_5)^t$ we see 
$$\left[\begin{matrix} \C f_1 \\ \vdots \\ \C f_5\end{matrix}\right] = 
\phi'(0) \left[\begin{matrix}  f_1 \\ \vdots \\ f_5 \end{matrix}\right] $$  if and only if 
$$(\C -\lambda I): f_1\mapsto f_2 \mapsto 0$$
 and
$$(\C-\alpha I): f_3\mapsto f_4 \mapsto f_5\mapsto 0.$$

Thus, a solution to Schroeder's equation exists if and only if $\C$ has a $\lambda$-chain of length (at least) 2, and an $\alpha$-chain of length (at least) 3.   Suppose for the moment that $\C$ has a $\lambda$-chain of length strictly greater than 2; for example, 
$$(\C -\lambda I): g_k \mapsto  ... \mapsto g_2 \mapsto g_1\mapsto  0.$$
Then only the last two elements of the chain can be used to solve Schroeder's equation; that is,  $f_2= g_2$ and $f_1=g_1$  are the only elements of the chain $\{g_i\}$ that can appear as the first two coordinate functions of $F$. This idea is made rigorous in Lemma \ref{S: three equiv. lemma}.
\end{example}

Putting $\phi'(0)$ into Jordan form provides a convenient set-up for the problem.  Specifically, it gives a basis of $\Cn$ that is a union of chains of $\phi'(0)$, and we now know that solutions to Schroeder's equation depend on finding the corresponding chains of $\C$.  The next step will be to use the compactness of $\C$ to reduce the problem to understanding the chains of a particular finite rank operator. (It is denoted $U$ throughout the paper).  Representing this operator as an $N\times N$ matrix,  and then putting it into Jordan form will yield it's chains.  So Jordan form not only simplifies the problem, but it plays a role in the main proofs of the paper.  

\section{Notation}
\label{Sec: N}
In the spirit of the previous section, we now let $\phi$ be a self-map of $\D$, an ellipsoid  (that is, $\D = D\Bn$ for some invertible matrix $D$) fixing $0$, and having $\phi'(0)$ an invertible, upper-triangular, Jordan-form matrix.  Further, we suppose the composition operator $\C$ is compact on $A_H^2(\D)$.  For convenience we will often represent elements $g$ of $\A$ as (infinite) column vectors and $\C$ as an (infinite) matrix.  To do this, given such a $g$, we know $g=\sum_{\alpha}a_{\alpha}z^{\alpha}$ with absolute convergence on $\D$.  Thus, we write 
$$g=\left[ \begin{matrix} a_{(0,0,...,0)}\\
a_{(1,0,...,0 )}\\
a_{(0,1,...,0)}\\
\vdots
\end{matrix} \right].$$  It is important to note that since we have proven only that $\sum_{j=0}^n g_j \to g$ in $\A$ with $g_j$ homogeneous polynomials of degree $j$, we understand $g=\sum_{\alpha}a_{\alpha}z^{\alpha}$ as the sub-sequential limit of homogeneous polynomials when dealing with convergence in the norm of $\A$.  

We assume the dictionary ordering on the multi-indices; specifically, ${\alpha}<{\beta}$ if and only if either $ | \alpha|<| \beta |,$ or $ | \alpha|=| \beta |$ and writing $\alpha=(\alpha_1,...,\alpha_n), \beta= (\beta_1,...,\beta_n)$ we also have $\alpha_j=\beta_j$ for $j<j_0$, and $\alpha_{j_0}>\beta_{j_0}$ for some $1\leq j_0 \leq n$.  
  
Our matrix representation of $\C$ follows from this ordering of $\{z^{\alpha}\}_{\alpha}$.  If the $j^{th}$ monomial is $z^{\alpha}$, then the $j^{th}$ column of $\C$ will be the column vector $ \phi^{\alpha}$.  For example, if we are working in $\mathbb{C}^2$ and $\phi(z)=(\phi_1,\phi_2)$ with $\phi_1(z)=\lambda_1 z_1 +\lambda_2 z_2 + \sum_{|\alpha |>1}a_{\alpha}z^{\alpha},$ and $\phi_2(z)=\zeta_2 z_2 + \sum_{|\alpha |>1}b_{\alpha}z^{\alpha}$, we see $\C: 1\mapsto 1, z_1 \mapsto \phi_1, z_2 \mapsto \phi_2$ and so forth.  Thus, 
$$ \C = \begin{array}{cc} 
  \hphantom{a}  & \begin{matrix}  & \hphantom{a}1	&\hphantom{.}\phi_1\hphantom{a}	& \phi_2\phantom{a}   & \phi_1^2\phantom{aa} & \hphantom{\dots} \end{matrix}\\
\begin{matrix}
1\\
z_1\\
z_2\\
z_1^2\\
\vdots
\end{matrix}&
					\left[ \begin{matrix}
					1 		& 0 		& 0 		& 0 	& \dots  \\
					0  		& \lambda_1	&  0		& 0			&\dots \\
					0		& \lambda_2	& \zeta_2	& 0			& \dots\\
					0		& a	_{1,0}	& b_{1,0}	& \lambda_1^2 & \dots \\
					\vdots	& \vdots	&\vdots		&\vdots 	&\ddots 

\end{matrix} \right]
\end{array}.$$

Notice that the first row and column of $\C$ will always be $(1,0, ... )$ or $(1,0, ... )^t$, respectively, so we will omit these when representing $\C$.  In other words, we will consider $\C$ acting on the  subspace of functions which fix 0.  With this omission, we see the upper left $n\times n$ corner of $\C$ is $\phi'(0)^t$.  Lastly, we note that $\phi'(0)$ upper triangular implies $\C$ is lower triangular, which is exemplified above. 

Other important observations are that  the diagonal entries of $\C$, $c_{j,j}=\partial \phi_j(0)/\partial z_j$ 
for $j=1,...,n$, and in general,  $c_{j,j}=(c_{1,1},...,c_{n,n})^{\alpha}=c_{1,1}^{\alpha_1}...c_{n,n}^{\alpha_n}$ where $z^{\alpha}$ is the $j^{th}$ monomial.  This notation follows the literature.  A more thorough treatment can be found in \cite{CM03}.  

Since an upper triangular $\phi'(0)$ implies $\C$ is lower triangular, when putting $\phi'(0)$ into Jordan form, we will opt for the upper triangular version.  As $\phi'(0)^t$ appears in $\C$, we can see that lower triangular Jordan blocks will appear on the left side of Schroeder's equation while their transposes, upper triangular Jordan blocks, appear on the right side.  For notational ease we agree that  $J$ (or often $J_i$ with $i\in \mathbb{N}$) will denote a lower triangular Jordan block throughout the paper, and we will use $J^t$ to denote the upper triangular counterpart.  Recall that a Jordan matrix is a matrix with either all upper triangular  or all lower triangular Jordan blocks along the diagonal, and all other entries 0.  In particular, the Jordan matrix $\phi'(0)$ will be written in block form as 
$$ \phi'(0)=\left[ \begin{matrix} J_1^t \\ & \ddots & J_m^t \end{matrix}\right]$$
with each $J_j^t$ an upper triangular Jordan block. We call the subspace on which $J_i^t$ acts, say $S_i$, a Jordan subspace.  This is made rigorous in the definition below.    
\begin{definition}
Given a linear operator $L$ on a vector space $X$, we say $S$ is a Jordan subspace of $L$ with eigenvalue $\lambda$ if (and only if) 
\begin{enumerate}
\item $X=S\oplus T$ with $T$ an algebraic complement of $S$. 
\item $L=\left[\begin{matrix} J & 0\\ 0 & K
\end{matrix}\right],$ with $J$ and $ K$ the compressions of $L$ to $S$ and $ T$, respectively.  
\item $J$ is a $\lambda$-Jordan block. 
\end{enumerate}
   We will say $J$ has length $k=\dim(S)$. 
 \end{definition}
  Since $J$ is lower triangular there is a basis $e_1, ... , e_k$ for $S$ such that 
\begin{equation}
\label{P: e1}
  (L-\lambda I) e_j=\left\{ \begin{array}{cc}
   e_{j+1} &\mbox{ if } 1\leq j \leq k-1\\
   0 & \mbox{ if } j=k,
   \end{array}\right.
\end{equation}   
 or concisely, $e_1, ... , e_k$ is a chain basis for $S$, with $e_k$ as the eigenvector.  Reversing this order of this chain shows the definition is unaffected by using $J^t$ in place of $J$.  We remark that $k$ is the maximal length of a chain of $L$ with eigenvector $e_k$, and also that $K$ may have another chain with eigenvalue $\lambda$.  For example, 
 $$\begin{matrix}e_1 \\ e_2 \\ e_3 \end{matrix} \left[\begin{matrix} 
 \lambda & 1 		&0\\ 
		0& \lambda 	&0\\
		0&	0		& \lambda 
 \end{matrix}\right]$$ 
 has $\langle e_1, e_2 \rangle$ and $\langle e_3 \rangle $ both as Jordan subspaces.  Lemma \ref{LA: chain lin. ind lemma}, which states that a set of $m$ chains with linearly independent eigenvectors is also linearly independent, will be used even before it is proven in Section \ref{Sub Sec: LA}.    

In light of the above discussion, our goal is now to determine necessary and sufficient conditions for the existence of a full rank solution for $\phi$ a self-map of the ellipsoid $D\Bn$ fixing $0$, $\phi'(0)$ an invertible, upper-triangular,  Jordan matrix, and $\C$ compact on $A_H^2(\mathcal{D})$.     
We now have a Jordan decomposition  induced by $\phi'(0)^t,$ 
$$\Cn =S_1 \oplus...\oplus S_m,$$ 
and  $\phi'(0)^t=$ diag$(J_1, ... , J_m)$ 
 with each $J_j$ a lower-triangular Jordan block.  That is, the compression of $\phi'(0)^t$ to $S_j$ is $J_j$, and $S_j$ is a Jordan subspace of $\phi'(0)^t.$  Let $n_j:=\dim(S_j),$ so $J_j$ is $n_j \times n_j$  with eigenvalue $\lambda_j$.  It follows that $n_1+ ... + n_m=n$.  Notice also that the $\lambda_j$ need not be distinct.

   We will let the chain $e_1^j, ... , e_{n_j}^j$ be  our basis of $S_j$, so that 
   $$\{e_1^1, ... , e_{n_1}^1, ... , e_1^m, ..., e_{n_m}^m\}$$
    is our full basis on $\Cn$.  Quite explicitly, we now have 
    $$J_j=\begin{matrix} 
 e_1^j\\ e_2^j\\ \vdots \\ e_{n_j}^j   
    \end{matrix}\left[ \begin{matrix}
 \lambda_j 	& 			&\\
 1			&\lambda_j 	& \\
 			&\ddots 	& \ddots \\
 			&			&1		& \lambda_j
  \end{matrix}\right].$$  
 
   Just as we represent $\phi'(0)^t$ in block form, we also write $F=(F_1, ... , F_m)^t$ in block form, with each $F_j$ a column-vector with length $n_j$.  So, each 
   $$F_j=\left[\begin{matrix} f_1^j\\ \vdots \\ f_{n_j}^j \end{matrix} \right]$$ 
   is a function from $\mathcal{D}\to \mathbb{C}^{n_j}$ with each $f_i^j:\mathcal{D} \to \mathbb{C}$.  In other words, $F_1$ is the first $n_1$ components of $F$, then $F_2$ is the next $n_2$ components of $F$, and so forth.  Schroeder's equation becomes 
 $$\left[ \begin{matrix} F_1\circ \phi \\
  \vdots \\
  F_m\circ \phi 
 \end{matrix}  \right] = F\circ \phi=\phi'(0) F= 
\left[ \begin{matrix}
 J_1^t & \dots & 0\\
\vdots	 & \ddots& \vdots\\
0	&\dots & J_m^t
 \end{matrix}\right] 
 \left[ \begin{matrix} F_1 \\
  \vdots \\
  F_m
 \end{matrix}  \right] = 
 \left[ \begin{matrix} J_1^t F_1 \\
  \vdots \\
  J_m^t F_m
 \end{matrix}  \right]. $$
 Thus, a Schroeder solution exists exactly when we can solve  $F_j\circ\phi=J_j^tF_j$ for each $j$ and the function $F=(F_j)^t$ has full rank near 0.  
 
 Unpacking notation we write $F_j=(f_1^j, ... , f_{n_j}^j)^t, $ and see  $F_j\circ\phi=J_j^tF_j$ becomes 
 \begin{equation}
  \label{S: e1}
  \left[ \begin{matrix}f_1^j\circ\phi \\ \vdots \\ f_{n_j-1}^j\circ\phi \\ f_{n_j}^j\circ\phi   \end{matrix}\right]
 =\left[ \begin{matrix}
 \lambda_j 	& 1			& \dots   & 0\\
 0			& \ddots 	& \ddots   & \vdots \\
 \vdots		& \ddots	& \lambda_j  & 1 \\
0 			& \dots			& 0 		& \lambda_j
 \end{matrix}\right]
 \left[\begin{matrix}
 f_{1}^j \\ \vdots \\ f_{n_j-1}^j \\ f_{n_j}^j
 \end{matrix}\right]=
 \left[\begin{matrix} \lambda_j f_{1}^j +f_{2}^j \\
 \vdots \\
 \lambda_j f_{n_j-1}^j +f_{n_j}^j\\
 \lambda_j f_{n_j}^j \end{matrix} \right].
 \end{equation}
 In other words, $f_1^j, ... , f_{n_j}^j$ is a chain of $\C$ with eigenvalue $\lambda_j$.  

 \begin{lemma}
\label{S: three equiv. lemma}
The following statements are equivalent. 
\begin{enumerate}
\item There exists an $F_j=(f_1^j, ... , f_{n_j}^j)^t$ with components linearly independent in $\h$ satisfying $F_j\circ \phi =J_j^tF_j$.  
\item $[\ker(\C-\lambda_jI)] \cap [(\C-\lambda_jI)^{n_j-1}(\A)]$ is nontrivial.  
\item $\C$ has a chain of length $n_j$ and eigenvalue $\lambda_j$.  
\end{enumerate} 
 \end{lemma}
\begin{proof}
The work above shows that $F_j=(f_1^j, ... , f_{n_j}^j)^t$ satisfies $F_j\circ \phi = J_j^tF_j$ if and only if $\{f_i^j\}_i$ forms a chain of $\C$ with eigenvalue $\lambda$.  Lemma \ref{LA: chain lin. ind lemma} shows the $f_i^j$ are linearly independent.  Thus (1) and (3) are equivalent.  The equivalence of (2) and (3) is the definition of a chain. 
\end{proof}
We have reduced the problem of finding an $F$ satisfying $F\circ \phi=\phi'(0) F$ to that of finding chains of $\C$ which correspond to those of $\phi'(0)$.  We now are ready to produce solutions (Section \ref{Sec. Sw/oCRN0}), and then discriminate which have full rank near 0 (Sections \ref{Sec: R} and \ref{Sec: CfaFRS}).  

\section{Solutions Without Consideration of Rank Near 0}
\label{Sec. Sw/oCRN0}
 
 Forgetting the behavior of $F$ near 0, we focus on a much more modest goal of producing a solution, $F$, whose component functions are linearly independent in $\A$.  Unlike a full rank solution, a linearly independent solution necessarily exists (Theorems \ref{S: t3} and \ref{S: c3}), yet, the process by which we develop a linearly independent solution will provide the groundwork for understanding exactly when  we are equipped with a full rank solution.  Specifically, we will use compactness of $\C$ to reduce the problem to one of linear algebra, and Subsection \ref{Sub Sec: LA}, a section of linear algebra theorems, provides the lion's share of work for both Sections \ref{Sec. Sw/oCRN0} and \ref{Sec: CfaFRS}.

The main technique in \cite{CM03}, as well as the current note is to exhibit a Bergman space on which $\C$ is compact, and then use the compactness to produce our solutions.  The following lemma and corollaries roughly say that any compact operator when viewed as an infinite matrix acts like a sufficiently large (finite) upper left corner.  Consequently, we may reduce the problem of existence of a solution $F$ to an analogous problem for an operator on $\CN$.  

\begin{theorem}
\label{compactness theorem}
 Let $C$ be a lower-triangular compact operator on a separable Hilbert space, $\mathcal{H}$.  Given any $\lambda \neq 0$ the following are true.  
 \begin{enumerate}
 \item We can write $\mathcal{H}=H_1 \oplus H_2$ with $H_1\cong \mathbb{C}^N$  and $$C= \left[ \begin{matrix}
 U&0\\
 V&W
\end{matrix} \right]$$ 
 so that $|\lambda| > \| W\|$.
\item $\ker(C-\lambda I) \cong \ker(U-\lambda I_N)$, where $U$ is $N\times N$. 
\end{enumerate}
\end{theorem}

It follows easily from this theorem that $\sigma(C)\setminus \{0\} = $ diag$(C)\setminus \{0\}$, where $\sigma(C)$ denotes the spectrum of $C$. As these results are well known, we omit the proof.  The details may be found in \cite[Lemma 9]{CM03}.  The following corollary states the ideas of Theorem \ref{compactness theorem} in a way that is most useful for our purposes with explanations in the paragraphs that follow.      

 \begin{corollary}
 \label{A:l1}
 Let $C$ be a lower-triangular compact operator on a separable Hilbert space, $\mathcal{H}$.  Given any $\lambda \neq 0$ the following are true.  
 \begin{enumerate}
 \item We can write $\mathcal{H}=H_1 \oplus H_2$ with $H_1\cong \mathbb{C}^N$  and $$C= \left[ \begin{matrix}
 U&0\\
 V&W
\end{matrix} \right]$$ 
 so that $\lambda \notin \sigma_p(W)$, the eigenvalues of $W$.
\item $\ker(C-\lambda I) \cong \ker(U-\lambda I_N)$, where $U$ is $N\times N$. 
\item Any formal solution, $h$ to $(C-\lambda I)h=0$ is in $\h$.
\end{enumerate}
\end{corollary}

Part (1) of the corollary below is stated for usability. In practice, given the lower triangular matrix $\C$ and a fixed $\lambda$, it is much easier to choose the upper left corner $U$ so that it contains every occurrence of $\lambda$ on the diagonal (recall that $\C$ is compact so the diagonal entries, i.e. its eigenvalues, converge to 0) than it is to choose $U$ so that $\| W\|<|\lambda|.$  As in the results of \cite{CM03}, to prove our main theorems we will choose $U$ to include every eigenvalue of  $\phi'(0)$ that appears on the diagonal of $\C$.  

The important consequence of (2) is that given any 
$$h_1=\left[ \begin{matrix}x_1 \\ \vdots \\x_N\end{matrix}\right] \in \ker(U-\lambda I_N),$$
 there exists a unique completion of this vector, that is, a unique 
$$h_2=\left[\begin{matrix}x_{N+1}\\ \vdots \end{matrix}\right]\in H_2 $$  
so that 
$$h_1+h_2=\left[\begin{matrix}x_1 \\ \vdots \\x_N \\ x_{N+1} \\ \vdots \end{matrix} \right] \in \ker(\C-\lambda I).$$
 
By a formal solution, we mean a column vector $h=(x_1, x_2, ... )^t$ which satisfies the equation $C h=\lambda h$.  No concern is given to the convergence of $h=\sum x_je_j$ in $\h$.  To see the use of (3), consider writing the equation $C h=\lambda h$ with $C$ as a given lower triangular matrix and $\lambda$ also fixed.  It gives an (infinite) system of equation, the first equation depending only on $x_1$, the second depending on $x_1$ and $x_2$ and so on.  If $x_j$ can be chosen inductively to satisfy our algebraic requirements, then (3) tell us that $\sum x_je_j$ converges in $\h$. In particular, when $\h$ is a Hilbert space of analytic functions on a domain $\mathcal{D}$ we see the $x_j$ are the coefficients of the Taylor series of an analytic function, and (3) says that this corresponding Taylor series actually converges on $\mathcal{D}$.  

We reiterate that the collection of ideas in Theorem \ref{compactness theorem} and Corollary \ref{A:l1} are well know and certainly the use of compactness of $\C$ to find analytic solution is evident in \cite{CM03}.  Perhaps it is the thorough investigation of these ideas that allowed proofs of the main results of this paper.  
As a last remark, we note that compactness is essential here.  For example, the right and left shift operators are lower and upper triangular, respectively, and have the closed unit ball as their spectrum.  In particular, the left shift has the entire open disk as it's point spectrum, but only 0 appears on their diagonal.  



Let $P$ be the orthogonal projection of $\h:=\A$ to the subspace spanned by those multi-indices of order less than or equal to $K$; that is, $$P \h= \langle z_1, z_2, ... , z_n, z_1^2, z_1z_2, ... , z_n^2, ... , z_1^K, z_1^{K-1}z_2, ... ,  z_n^K \rangle.$$  Now we may write 
\begin{equation}
\label{S: e1.5}
\C= \left[ \begin{matrix}
 U&0\\
 V&W
\end{matrix} \right]
\end{equation} so that  $U=P \C P$ with $K$ sufficiently large so that $\lambda_j$ does not occur on the diagonal of $W$ for every $j=1, ... , m$.  Following the notation above, let $H_1\oplus H_2= \h$ and $ H_1=P\h$ of dimension $N$, so that $U$ is $N\times N.$  
We also let $Q$ denote the orthogonal projection from $\h$ to the first $n$ coordinates.  That is,  $$Q\h=\langle z_1, ... , z_n\rangle,$$  and we note that $Qf=\nabla f (0)$ for any $f$ in $\h$. Since $\p$ is the upper left $n\times n$ corner of $\C$, we have  $\phi'(0)^t=Q\C Q$, and $$\langle z_1, ... , z_n\rangle =S_1\oplus ... \oplus S_m,$$ our Jordan decomposition of $\p$.  Recall that the $S_j$ are not necessarily orthogonal subspaces of $\h$, so we let $Q_j$ be the projection to $S_j$ along $(\oplus _{i\neq j}S_i)\oplus (I-Q)\h$.  In particular,  $Q=Q_1+...+Q_m$.  

Recall that a solution depends on the chains of $\C$.  Now we use compactness to reduce the problem to understanding the chains of $U$, our sufficiently large upper left corner of $\C$.  For notational ease, we will use $``e_j"$ to denote chains of $\phi'(0)$ or $\phi'(0)^t$ (these are elements of $\Cn$), and we will use $``\E_j"$ to denote chains of $U$ (these are vectors in $\CN$).    
\begin{lemma}
\label{S: U, C chain correspondence}
With the notation developed just above, suppose $0\neq \lambda$ and $\lambda \notin \sigma(W)$.  If $\E_1, ..., \E_{s}$ is a chain of $U$ with eigenvalue $\lambda$,  then there is a unique chain $g_1, ... , g_s$ of $\C$ with the same eigenvalue such that $Pg_j=\E_j.$  Conversely, if $g_1, ... , g_s$ is a chain of $\C$ with eigenvalue $\lambda$, then $Pg_j=\E_j$ is a chain of $U$ with the same eigenvalue.  
\end{lemma}
\begin{proof}
Assume first that $\E_1, ..., \E_{s}$ is a chain of $U$ with eigenvalue $\lambda$; that is 
\begin{equation}
\label{equation U chain}
(U-\lambda I_N): \E_s \mapsto ... \mapsto \E_1 \mapsto 0.
\end{equation}
An application of Corollary \ref{A:l1} furnishes the existence of a unique $g_1$ such that $Pg_1=\E_1$, and $(C-\lambda I)g_1=0$. Inductively, suppose that for each $1\leq i\leq s-1$ there are unique $g_i$ such that $Pg_i=\E_i$ and
\begin{equation}
\label{S: e3}
(\C-\lambda I)^k g_i=\left\{ \begin{array}{rr}
0 &\mbox{ if } k\geq i\\
g_{i-k} &\mbox{ if } k<i.
\end{array}\right. 
\end{equation}
Now write $(\C-\lambda I)^s=C+(-\lambda)^sI$ with $C=(\C^s-s\lambda \C^{s-1}+... +s(-\lambda)^{s-1}\C$).  Since $\C$ is compact and lower triangular, so also is $C$.  An application of  Corollary \ref{A:l1} to $C+ (-\lambda I)^sI$ produces a unique $g_s$ satisfying $Pg_s=\E_s$ and $(\C-\lambda I)^sg_s =(C+ (-\lambda I)^sI)g_s=0$.  Next for any $1\leq k \leq s-1$ we see $P(\C-\lambda I)^kg_s =(U-\lambda I_N)^k\E_s=\E_{s-k}$.  Uniqueness of $g_{s-k}$ now implies that  $g_{s-k}=(\C-\lambda I)^kg_s$ completing our induction step.  

For the converse suppose $g_1, ... , g_s$ form a chain of $\C$ with eigenvalue $\lambda$.  Clearly, $\E_j:=Pg_j$ satisfies (\ref{equation U chain}).  We must show $\E_1\neq 0$. But if $\E_1=0$, then $\lambda g_1=\C g_1= W g_1$, a contradiction to $\lambda \notin \sigma(W)$. 
\end{proof}  

\begin{theorem}
\label{S: three equiv. Theorem}
With notation as above the following statements are equivalent. 
\begin{enumerate}
\item There exists an $F_j=(f_1^j, ... , f_{n_j}^j)^t$ with components linearly independent in $\h$ satisfying $F_j\circ \phi =J_j^tF_j$.  
\item $[\ker(U-\lambda_jI)] \cap [(U-\lambda_jI)^{n_j-1}(\h)]$ is nontrivial.  
\item $U$ has a chain of length $n_j$ and eigenvalue $\lambda_j$.  
\end{enumerate}
Moreover, when these three conditions hold, $\E_i^j:=Pf_i^j$ is the the chain in \mbox{(3)} and a chain $\E_i^j$ of $U$ uniquely determines such an $F_j=(f_i^j)^t$.    
\end{theorem}
\begin{proof}
It follows from Lemma \ref{S: three equiv. lemma} that $F_j$ is a solution if and only if the $f_i^j$ form a chain.  Suppose first $f_1, ..., f_{n_j}^j$ is a chain of $\C$ with eigenvalue $\lambda_j$, then clearly $Pf_i^j$, $i=1, ..., n_j$ is a chain of $U$.  Conversely, given a chain of $U$, say $\E_1^j, ... , \E_{n_j}^j$ with eigenvalue $\lambda$, Lemma \ref{S: U, C  chain correspondence} shows that there is a unique solution $F_j=(f_1^j, ... , f_{n_j}^j)^t$ with $PF_i^j=\E_i^j$. Furthermore, Lemma \ref{LA: chain lin. ind lemma} shows that the $f_i^j$ are linearly independent.  This shows $(1)\Leftrightarrow (3)$. 

$(2)\Leftrightarrow (3)$ is trivial. 
\end{proof}
 The preceding results identify the existence of such an $F_j$ with the chains of $U$.  Notice that if $U$ has a Jordan subspace of dimension at least $n_j$, then $[\ker(U-\lambda_jI)] \cap [(U-\lambda_jI)^{n_j-1}(\h)]$ is indeed non-empty (and the desired $F_j$ exists).  We now turn our focus to the Jordan subspaces of $U$ and state this in the next corollary.   
 
 Let $\zeta_1,...,\zeta_l$ be the eigenvalues of $U$.  Since $U$ is similar to a Jordan form matrix, $U$ has a  Jordan decomposition,
$$\mathbb{C}^N=\tilde S_1\oplus ... \oplus \tilde S_l.$$  
As usual, the Jordan subspace $\tilde S_k=\langle \{\E_i^k\}_i\rangle $ with $1\leq i \leq \dim(\tilde S_k)$ and $\E_i^k$ a chain of $U$ with eigenvalue $\zeta_k$.  
Notice that each Jordan subspace of $U$ contains exactly one eigenvector of $U,$ and partially conversely, any eigenvector of $U$ can lie in at most one $\tilde S_k$.  
\begin{corollary}
\label{S: c1}
With notation as above, if there exists an $\tilde S_k$ with $\zeta_k=\lambda_j$ (that is, the eigenvalue of $S_k$ and the $J_j$ coincide) and $\dim(S_k)\geq n_j$, then there is an  $F_j=(f_i^j)^t$ satisfying $\C F_j=J_j^tF_j$ with each linearly independent components. 
\end{corollary}
In fact, the converse holds as well, giving necessary and sufficient conditions for such an $F_j$. See Theorem \ref{S: promised converse}.
\begin{proof}
If there exists an $S_k$ with dimension at least $n_j$ and eigenvalue $\lambda_j$, we see it's eigenvector lives in   $[\ker(U-\lambda_j)] \cap [(U-\lambda_j)^{n_j-1}(\h)]$, which gives the result.  
\end{proof}
  We are led to the following question: 

Suppose we are given an $N\times N $ lower triangular matrix, $U$, with  $\phi'(0)^t$ the upper left $n\times n$ corner.  That is,
\newcommand{\BigFig}[1]{\parbox{12pt}{\huge #1}}
\newcommand{\BigY}{\BigFig{$Y$}} 
\newcommand{\BigI}{\BigFig{0}}
\newcommand{\BigZ}{\BigFig{$Z$}}
$$U=\begin{matrix} e_1 \\ \vdots \\e_n \\  e_{n+1} \\ \vdots\\ e_N  
\end{matrix}  \left[\begin{array}{c|c}
\begin{matrix}  
J_1 \\
	& \ddots\\
	&		& J_m   
\end{matrix} 		&  \begin{matrix}  
								& & \\ &  \BigI & \\
								& &  \\
								\end{matrix}\\
\hline
\begin{matrix} \\ \BigY \\ \\ \end{matrix} &
\begin{matrix} \\ \BigZ \\ \\ \end{matrix}
\end{array}\right].  $$
``Does each $J_j$ appear, possibly longer, in the Jordan form of $U$?"
 Equivalently, ``Does there exist an injective map identifying each $S_j$ (of a Jordan subspace decomposition for $\phi'(0)^t$) with an $\tilde S_k$ (of a Jordan subspace decomposition for $U$) that shares the same eigenvalue as  $S_j$ and has possibly larger dimension?" Or in terms of chains, ``For each $j$ does there exist a chain,  $\{\E_1^j, ... , \E_{n_j}^j\}$,  of $U$ with eigenvalue $\lambda_j,$ such that 
 $$\bigcup_{i,j}\{\E_i^j\}$$
 is a linearly independent set?"    An affirmative answer will ensure the existence of a solution, $F$ with linearly independent component functions.  What follows is a subsection of linear algebra results that  lead to this answer and are the basis of every major theorem of the paper.

\subsection{Linear Algebra Results}   
\label{Sub Sec: LA}
Certainly some or all of the results in this section are known in one form or another.  We develop and prove them here in a manner that is useful for our purposes. 
\begin{lemma}
\label{LA: chain lin. ind lemma}
Let $L$ be a linear operator on a vector space $V$.  Let $e_1^j, ... , e_{n_j}^j$  be a chain of $L$ with eigenvalue $\lambda_j$ for each $1\leq j \leq m$, and suppose the eigenvectors, $e_{n_1}^1, ... , e_{n_m}^m$ are linearly independent.  Then $\{e_i^j: 1\leq j \leq m,  1\leq i \leq n_j \}$ is a linearly independent set.  
\end{lemma}
\begin{proof}
We induct on $n=n_1+ ... +n_m=| \{e_i^j: 1\leq j \leq m,  1\leq i \leq n_j \}|.$  The result is clear for $n=1$.   Letting $n>1$, we write
\begin{equation}
\label{la: e1}
\sum_{j=1}^m \sum_{i=1}^{n_j} a_i^j e_i^j=0
\end{equation} 
for constants $a_i^j$, and show that $a_i^j$ are necessarily 0.  Applying $L-\lambda_1 I$ gives
$$\sum_{i=1}^{n_1-1}a_i^j e_{i+1}^j + \sum_{j=2}^m\sum_{i=1}^{n_j} b_i^j e_i^j=0$$
 for some constants $b_i^j$.  
 Since $e_2^1, ... , e_{n_1}^1$ is still a chain of $L$, our inductive hypothesis implies $a_i^1=0$ for any $1\leq i \leq n_1-1$.  Repeating this argument, we have $a_i^j=0$ for any $1\leq i \leq n_j-1$, and (\ref{la: e1})
 becomes 
 $$ \sum_{j=1}^ma_{n_j}^je_{n_j}^j=0.$$ 
 The result follows from the linear independence of our eigenvalues, $e_{n_1}^1, ... , e_{n_m}^m. $ 
\end{proof}

\begin{lemma}
\label{S:l2}
\label{LA: one J block lemma}
Let $J$ be a lower triangular $\lambda$-Jordan block, of length $k$. Set 
\newcommand{\BigJ}{\BigFig{$J$}} 
$$A=\left[  \begin{matrix} 
  & 0 \\
 \BigJ & \vdots \\ 
   & 0 \\
 a_1 \dots a_k & a_{k+1}\\
 \end{matrix} \right]. $$
If $B$ is the Jordan-form matrix conjugate to $A$ then either $B$ has two Jordan blocks, namely $J$ and the $1\times 1$ matrix, $[a_{k+1}],$ or $B$ is a $\lambda$-Jordan block of length $k+1$. More precisely, the later case happens if and only if $a_{k+1} = \lambda$ and $a_k \neq 0.$
\end{lemma}  
\begin{proof}
Let $\{e_1, ... , e_{k+1}  \}$ be the usual basis on $\mathbb{C}^{k+1}$.  We will exhibit a new basis, $\{\E_1, ..., \E_{k+1}\}$ giving the desired result. 

Suppose first that $a_{k+1}\neq \lambda. $  Set $\E_{k+1}=e_{k+1}$, and $\E_j=e_j+c_je_{k+1}$ with $c_k=a_k/(\lambda-a_{k+1})$ and $c_j=(c_{j+1}-a_j)/(a_{k+1}-\lambda)$. It follows that $\E_k$ and $\E_{k+1}$ are eigenvectors with eigenvalues $\lambda$ and $a_{k+1}$, respectively, and  $A(\E_j)=\lambda \E_j+\E_{j+1}$ for each $1\leq j \leq (k-1)$. So $B= $ diag$(J,a_{k+1}).$
     
Now let $a_{k+1}=\lambda$.  Set $\E_{k+1}=e_{k+1}$,  $\E_1=e_1$, and $\E_j=e_j+a_{j-1}e_{k+1}$ for each $2\leq j \leq k$. 
It follows that $A(\E_j)=\lambda \E_j+\E_{j+1}$ for each $1\leq j \leq (k-1)$. 
If $a_k=0$, we see $\E_k$ and $\E_{k+1}$ are linearly independent  eigenvectors, and
again we have $B=$ diag$(J,a_{k+1})=$ diag$(J,\lambda).$
 On the other hand, if $a_k\neq 0, A(\E_k)=\E_{k+1}$ and $\E_{k+1} $ is the lone eigenvector. It follows that $B$ is a $(k+1) \times (k+1)$ $\lambda$-Jordan block. 
\end{proof}
\begin{lemma}
\label{S:l3}
\label{LA: 2 J blocks lemma}
\newcommand{\0}{\BigFig{$0$}} 
\newcommand{\I}{\BigFig{$J_1$}}
\newcommand{\J}{\BigFig{$J_2$}}
Let $$A= \begin{matrix} e_1 \\ \vdots \\ e_n \\ e_{n+1} \\ \vdots \\ e_{n+k}\\ e_{n+k+1}
\end{matrix}
 \left[ \begin{array}{c|c|r}

 \I			& \0 		&	\begin{matrix}
							0\\ 
						\vdots\\
							0							 
							\end{matrix}\\
\hline
 \0 		& \J		&	\begin{matrix}
							0\\
							\vdots\\
							0\\
							\end{matrix}\\
\hline
\begin{matrix}
0 & \dots & 0 &1						
\end{matrix} 									& \begin{matrix}
													0 & \dots & 0&1 
													\end{matrix}								& \lambda 
 \end{array} \right] $$
 so that the last row of $A$ has a 1 in columns $n$ and $n+k$, $\lambda$ in column $n+k+1,$ and all other entries 0.  $J_1$ and $J_2$ are both lower triangular $\lambda$-Jordan blocks of length $n$ and $k$, respectively with $k\geq n$.  Then the Jordan-form matrix conjugate to $A$ consists of two $\lambda$-Jordan blocks of length $n$ and $k+1$ respectively.
\end{lemma}
\begin{proof}
Set $\E_1=e_1-e_{k+1}, \E_2=e_2-e_{k+2}, ... , \E_n=e_n-e_{n+k}$, and let $\E_{n+j}=e_{n+j}$ for $1\leq j \leq (k+1)$.  Notice that $\{\E_1, ... , \E_{n+k+1}\}$ is linearly independent.  Expressing $A$ in this basis gives the result with $\langle \E_1, ... , \E_n \rangle$ as the Jordan subspace with dimension $n$, and 
$\langle \E_{n+1}, ... ,  \E_{n+k+1}\rangle$ the Jordan subspace with dimension $k+1$. 
\end{proof}

\begin{theorem}
\label{S: t2}
\label{LA: J block theorem}
 Let $C=$diag$(J_1, ... , J_m)$
be an $n\times n$ matrix with each $J_j$ a lower triangular Jordan block with eigenvalue $\lambda_j$.  
Let $$A=  \left[ \begin{matrix}
	C &	0\\
 	D &	E
 \end{matrix} \right]$$ 
an $N\times N$ lower triangular matrix and let $B$ be a Jordan form matrix similar to $A$.  Then for every Jordan block, $J_j$, of $C$, there is a corresponding Jordan block, $\tilde J_j$ of $B$  sharing the eigenvalue $\lambda_j$, and such that length$(J_j) \leq $ length$(\tilde J_j)$.  
\end{theorem} 
\begin{proof}
Inductively, it suffices to show the result for the case $D$ is an $N\times1$ row vector and $E=(\lambda)$ is $1\times 1$, i.e. 
\newcommand{\largec}{\BigFig{$C$}}
$$A= \left[ \begin{array}{ccc|r}
 	&   &  	  & 0 \\
    & \largec &     & \vdots\\
	&   & 	  & 0 \\
\hline
a_1 		& \dots  & a_N	  & \lambda
 \end{array} \right].$$

 As above let $J_j$ be $n_j\times n_j$ with eigenvalue $\lambda_j$, and $\{e_1, ... , e_N\}$ the usual basis on $\mathbb{C}^N$.   
 
 Focus first on $J_1$.  
 In the case $\lambda_1 \neq \lambda$ or the case $\lambda_1 = \lambda$  and $a_{n_1}=0$, the proof of Lemma \ref{S:l2} allows us to assume $a_1= ... =a_{n_1}=0$   by replacing $e_j$ with $\E_j=e_j+c_je_{N+1}$ for   $1\leq j \leq n_1$ with appropriate values of $c_j$.    If, on the other hand, we have $\lambda_1=\lambda$ and $a_{n_1} \neq 0$, a similar application of Lemma \ref{S:l2} gives the simplification $a_1=...=a_{n_1-1}=0,$ and $a_{n_1}=1$ with no loss of generality. 

Consequently, it suffices to show the result under the assumption that each $J_j$ is a $\lambda$-Jordan block, and furthermore that 
$$a_i=\left\{ \begin{array}{cc}
0 &\mbox{if the $i^{th}$ column is not the last column of some $J_j$}\\
1 &\mbox{if the $i^{th}$ column is the last column of some $J_j$.}
 \end{array}\right.$$
Explicitly,
\newcommand{\0}{\BigFig{$0$}}
\newcommand{\I}{\BigFig{$J_1$}} 
\newcommand{\J}{\BigFig{$J_m$}}
$$A= \left[ \begin{array}{c|c|c|r}

 \I			&\dots			& \0 		&	\begin{matrix}
							0\\
						\vdots\\
							0					\end{matrix}\\
\hline
\vdots		& \ddots	& \vdots &\vdots\\
\hline
 
 \0 		& \dots		& \J		&	\begin{matrix}
										0\\
										\vdots\\
										0 
										\end{matrix}\\
\hline
\begin{matrix}
0			 & \dots 	& 0		 &1 						
\end{matrix} 									& \begin{matrix}
													\dots 
													\end{matrix}					&   \begin{matrix}
																						0 & \dots & 0 &1
																						\end{matrix}			& \lambda 
 \end{array} \right]$$ 
By using a unitary change of basis, arrange the Jordan blocks so the $n_i=$ length$(J_i) \leq n_j=$ length$(J_j)$ for $1\leq i\leq j\leq m$.  Inductive application of \ref{S:l3} will give the appropriate change of basis to attain the Jordan form matrix, $B$, and moreover, each $J_i$ will appear unaffected in $B$ except for $J_m$ which will now appear longer.  
\end{proof}
We are now able to answer our question.  When $A=U$ and $C=\phi'(0)^t=$ diag$(J_1, ... , J_m)$, the above result says that there is a basis so that expressing $U$ in this basis gives the Jordan matrix, $U=$ diag$(\tilde J_1, ... , \tilde J_l)$
 with $l\geq m$.  Moreover, for every $1\leq j\leq m$, length$(J_j)\leq$ length$(\tilde J_j)$  and the eigenvalue of $\tilde J_j$ is $\lambda_j$.   
\begin{corollary}
\label{S: c2}
\label{LA: decomp corollary}
Let 
$$\Cn=\oplus_{j=1}^m S_j \mbox{ and } \mathbb{C}^N=\oplus_{j=1}^l\tilde S_j$$ 
be Jordan subspace decompositions of $\p$ and $U$, respectively.  Then there is an injection,  
$$\tau : \{S_j\}_{j=1}^m \to \{\tilde S_j\}_{j=1}^l, $$
 such that the eigenvalue of $S_j$ is preserved, and $\dim(S_j)\leq \dim(\tau S_j).$
\end{corollary}
The following theorem pays close attention to the proof of Theorem \ref{LA: J block theorem} and will be useful in distinguishing whether a solution with full rank near 0 exists or not. The notation introduced in the theorem below matches that developed earlier in the section.   
\begin{theorem}
\label{LA: workhorse theorem}
Let $\Cn=\oplus_{j=1}^m S_j$ be a Jordan subspace decompositions induced by $\p$, and  $S_j=\langle e_i^j: 1\leq i \leq n_j\rangle$ the usual chain basis.  Consider $\Cn$ as the subspace of $\CN$ spanned by the first $n$ components, and let $Q_j$ be the projection to $S_j$ along $(\oplus_{i\neq j}S_i)\oplus (\Cn)^{\perp}$. Then there exists a Jordan subspace decomposition induced by $U$, namely, $\mathbb{C}^N=\oplus_{j=1}^l\tilde S_j$ 
with $d_j=\dim(\tilde S_j)$ and chain bases $\langle \E_i^j: 1\leq i \leq d_j\rangle =\tilde S_j$ such that the following hold.  
\begin{enumerate}
\item $Q_j(\E_i^j)=e_i^j$ for any $1\leq j \leq m$, and $1\leq i \leq n_j$.  
\item The map $\tau(S_j)=\tilde S_j$ satisfies the result of Corollary \ref{S: c2}. 
\item $d_j>n_j$ for some $1\leq j \leq m$ if and only if $Q(\ker(U-\lambda_j I_N))\subsetneq \ker(\p-\lambda_j I_n)$.  
\end{enumerate}
\end{theorem}
\begin{proof}
The basis $\{ \E_i^j\}$ of $\mathbb{C}^N$ which will give our solution is the basis that is produced in the proof of Theorem \ref{S: t2}.  To prove this, we will prove our results hold at each step of the induction proof of Theorem \ref{S: t2}.
With no loss of generality, we will show it only for the first step.  Assume we have
$$U_{n+1}= \left[ \begin{array}{ccc|r}
 	&   &  	  & 0 \\
    & \p &     & \vdots\\
	&   & 	  & 0 \\
\hline
d_1 		& \dots  & d_n	  & \lambda
 \end{array} \right],$$ 
acting on $\Cn\oplus\mathbb{C}$, and denote our basis elements for $\Cn$ $\{e_i^j: 1\leq j \leq m, 1\leq i \leq n_j\}$ (in chains as above) and $e_{n+1}$ the extra basis vector.  

(1) In order to put $U_{n+1}$ into Jordan form, the basis elements are replaced via applications of Lemmas \ref{S:l2} and \ref{S:l3}.  In both of these Lemmas' proofs we see a basis vector $e_i^j$ is always replaced with $e_i^j+x$ where $x$ is some vector satisfying $Q_jx=0$.  This proves (1).

(2) Notice that a chain of basis elements, say $e_1^j, ... ,e_{n_j}^j$ always is replaced by a new chain with the same eigenvalue and of possibly greater length.  This proves (2).

(3) When putting $U_{n+1}$ into  Jordan form, at most one chain can gain length.  Suppose first that some chain, say $e_1^j, ..., e_{n_j}^j$, gains length.  So $j$ is fixed for the moment.  It follows from the proofs of Lemmas \ref{S:l2} and \ref{S:l3} that the new chain has eigenvector $e_{n+1}$, and any other chain's original eigenvector, say $e_{n_k}^k$, is replaced with $e_{n_k}^k+c_ke_{n+1}+b_ke_{n_j}^j$ where $b_k$ is either 1 or 0, and $c_k$ some constant.  It follows that $\dim(Q[\ker(U_{n+1}-\lambda_j I_{n+1}) ])=\dim[ \ker( \p-\lambda_j I_n)] - 1$ and so 
$Q[\ker(U_{n+1}-\lambda_j I_{n+1})]\subsetneq \ker(\p-\lambda_j I_n).$  

Lastly, suppose now that no chain gains length.  Lemma \ref{S:l3} is not used, and from \ref{S:l2} it follows that each original eigenvector, $e_{n_j}^j$ is replaced with $e_{n_j}^j+ c_{j, n_j}e_{n+1}$.  
Consequently, $(Q\ker(U_{n+1}-\lambda_j I_{n+1}))= \ker(\p-\lambda_j I_n).$ 
\end{proof}

The following generalization of \ref{LA: J block theorem}
 is now easily within our grasp.
\begin{theorem}
\label{C: t1}
\label{LA: extra theorem}
Let $$A=\left[ \begin{matrix}
C & 0\\
D & E
\end{matrix}\right] \mbox{ or } \left[ \begin{matrix}
C & D\\
0 & E
\end{matrix}\right], $$
 an $N\times N$ matrix, with $C$ the upper left $n\times n$ corner of $A$ and $n\leq N$.  Let 
 $$\Cn=\oplus_{j=1}^m S_j \mbox{ and } \mathbb{C}^N=\oplus_{j=1}^l\tilde S_j$$ 
be Jordan subspace decompositions of $C$ and $A,$ respectively. Then there is an injection, 
$$\tau : \{S_j\}_{j=1}^m \to \{\tilde S_j\}_{j=1}^l, $$
 such that the eigenvalue of $S_j$ is preserved, and $\dim(S_j)\leq \dim(\tau S_j).$
\end{theorem}
\begin{proof}
We will show the result for $$A=\left[ \begin{matrix}
C & 0\\
D & E
\end{matrix}\right],$$  
as a proof for the case
$$\left[ \begin{matrix}
C & D\\
0 & E
\end{matrix}\right]$$
is similar.  
Let $M$ be an invertible matrix such that $MCM^{-1}$ is in Jordan form (lower triangular), and let $L$ be an invertible matrix such that $LEL^{-1}$ is lower triangular.  Set 
$$\tilde M =\left[\begin{matrix}
M&0\\ 0&L\end{matrix}\right],$$
so $$
\tilde MA\tilde M^{-1}=\left[ \begin{matrix}
MCM^{-1} &0\\
\tilde D & LEL^{-1}
\end{matrix}\right]$$ 
is a lower triangular matrix with upper left $n\times n$ corner in Jordan form.  It follows that $\{MS_j \}_{j=1}^m$ and $\{\tilde M \tilde  S_j\}_{j=1}^l$ are the Jordan spaces of $MCM^{-1}$ and $\tilde M A {\tilde M}^{-1}$, respectively, with unchanged eigenvalues and dimensions.  An application of Corollary \ref{S: c2} gives the result.  
\end{proof}

\subsection{Main Results of Section \ref{Sec. Sw/oCRN0} }

We can now state and prove the main theorems of the section but first we give the promised converse to Corollary \ref{S: c1}.  We will continue to use the notation developed throughout the section.  
\begin{theorem}
\label{S: promised converse}
There is an $F_j=(f_1^j, ... , f_{n_j}^j)^t$  with linearly independent components satisfying $\C F_j=J_j^tF_j$ if and only if $U$ has a Jordan space decomposition, $\CN=\oplus_{j=1}^l\tilde S_j$ with some $S_k$ of dimension greater than or equal to $n_j$, and with eigenvalue $\lambda_j$.  
\end{theorem}
\begin{proof}
As the converse is given by Corollary \ref{S: c1} we only provide the forward direction.  

Suppose $F_j$ is a solution with linearly independent component functions.  It follows that $\{\E_i=Pf_i^j : 1\leq i\leq n_j\}$ is a chain of $U$ with length $n_j$ and eigenvalue $\lambda_j$.  By Lemma \ref{LA: chain lin. ind lemma} we know the $\E_i$ are linearly independent, so we may extend  this chain to a basis of $\CN$.  Expressing $U$ in this basis we see $U$ is of the form 
$$\left[\begin{matrix} J & A \\ 0 & B \end{matrix}\right],$$  
with the Jordan block $J$ the compression of $U$ to $\langle \E_1, ..., \E_{n_j}\rangle.$  An application of Theorem \ref{LA: extra theorem} gives the result.  
\end{proof}
While we now have necessary and sufficient conditions for a solution with linearly independent components, our jubilation is postponed as the achievement is overshadowed by the main results of the section. 
\begin{theorem}
\label{S: t3}
\label{S: main theorem}
Let $\phi$ be an analytic map on the ellipsoid $\mathcal{D}$ fixing 0 with $\phi'(0)$ a full rank upper-triangular Jordan form matrix. Suppose also that $\C$ is compact on a Bergman space, $\A$. Then there is an analytic $F$ satisfying $F\circ\phi=\phi'(0)F$ such that the component functions of $F$ are linearly independent in $\A$.  
\end{theorem}
\begin{proof}
As before,
 $\phi'(0)^t=$diag$(J_1, ... , J_m)$,
and $\A=\h=H_1\oplus H_2$ with 
$$H_1=\langle z_1, z_2, ... , z_n, z_1^2, z_1z_2, ... , z_n^2, ... , z_1^K, z_1^{K-1}z_2, ... ,  z_n^K \rangle.$$ The projection to $H_1$ is denoted $P$, and $H_1$ is isomorphic to $\mathbb{C}^N$.  Further, we write $$\C= \left[ \begin{matrix}
 U&0\\
 V&W
\end{matrix} \right],$$  with $U=P \C P$ and $K$ sufficiently large so that $\lambda_j$ does not appear on the diagonal of $W$ for every $j=1, ... , m$.    

Let $\{S_j\}_{j=1}^m $ and $\{\tilde S_j\}_{j=1}^l$ be as in Theorem \ref{LA: workhorse theorem}, whose respective bases are the chains, $e_1^j, ... , e_{n_j}^j$  and $\E_1^j, ... , \E_{d_j}^j$.  Recall that $d_j=\dim(\tilde S_j) \\
\geq \dim(S_j)=n_j$ for each $1\leq j \leq m$, and that $\{\E_i^j: 1\leq i \leq d_j, 1\leq i \leq l\}$ is a basis for $H_1$.  

Lemma \ref{S: U, C chain correspondence} furnishes us with unique $F_j=(f_1^j, ... ,f_{n_j}^j)^t$ such that $\C F_j=J_j^tF_j$ and $Pf_{i}^j=\E_{d_j-n_j+i}^j$.  Since $\{\E_i^j: 1\leq i \leq d_j, 1\leq i \leq l\}$  is linearly independent, we see $\{f_i^j: 1\leq i \leq n_j, 1\leq i \leq l\}$ is, too.  
\end{proof} 
Using Section \ref{Sec:SA} we can reformulate Theorem \ref{S: main theorem} in terms of our original hypotheses for Schroeder's equation on the ball $\Bn$ and are rewarded with the following result.  
\begin{theorem}[Main Theorem of Section \ref{Sec. Sw/oCRN0}]
\label{S: c3}
\label{S: main corollary}
Let $\phi$ be a self-map of $\mathbb{B}^n $ fixing 0 with $\phi'(0)$ invertible and $\phi$ not unitary on any slice.  Then there exists an analytic $F$ satisfying $\C F=\phi'(0)F$ and having component functions that are linearly independent in $A_G^2(\Bn)$.
\end{theorem}
Unfortunately, the fact that $F$ has linearly independent component functions does not imply that $F(\{|z|<\epsilon\})$ is open in $\mathbb{C}^n$ for small $\epsilon$ (as is the case when $F'(0)$ is full rank).  For example, $(z_1, z_2)\mapsto (z_1,z_1^2)$  contains no open subset of $\mathbb{C}^2$ in its image, while  $(z_1, z_2)\mapsto (z_1,z_2^2)$ does.  Additionally, neither of these maps have full rank near 0, but both do have linearly independent component functions.  

\begin{example}
\label{S: ex 0}
\end{example}
Consider $\phi(z_1,z_2)=(z_1/2, z_2/4 + z_1^2/16)$ on $\mathbb{B}^2$.  Clearly $\phi(0)=0$ and $|\phi(z)|^2\leq |z_1|^2/4+4(|z_2|^2/16+|z_1|^4/16^2)<|z|$ for $0<|z|<1$.  Now,
$$\phi'(0)^t=\mbox{diag}(1/2, 1/4),$$
 and
$$\C=\left[ \begin{matrix} 
1/2 	& 0 	& 0 	& 0 &\dots \\
0 		& 1/4	& 0		& 0 & \dots\\
0		& 1/16	& 1/4 	& 0 & \dots \\
0		& 0		& 0		& 1/8 &\dots \\
\vdots 	&\vdots &\vdots & \vdots &\ddots
\end{matrix}\right],$$ 
so $\ker(\C-(I/4))=\langle z_1^2 \rangle.$  $\phi'(0)$ has two chains both of length 1, and eigenvalues 1/2 and 1/4, respectively.  $\C$ has a (1/2)-chain of length 1, and a (1/4)-chain of length 2, namely $16z_2\mapsto z_1^2\mapsto 0$.  Thus, $(az_1,bz_1^2)^t$ with $a$ and $b$ fixed complex numbers are the only solutions to Schroeder's equation.  Consequently, every such solution must map open subsets of $\mathbb{B}^2$ into a 1-dimensional manifold. This also implies that no full rank solution can exist, but certainly $(z_1,z_1^2)^t$ is a solution with linearly independent components.  

We invite the interested reader to consult \cite[\S 4]{CM03}, a section dedicated to similar examples.  

 This section showed that the existence of a solution, $F$, with linearly independent components was characterized by the Jordan subspaces of $U$ in relation to those $\phi'(0)^t$, and hence, that such a solution always exists (Theorem \ref{S: t3}).  While the harder problem of finding a full rank solution is expectedly less generous, in retrospect we know how to go about it.  First we find a the appropriate chains of $U$ (or equivalently $\C$) and then we check to see if their projections to $\langle z_1, ... , z_n \rangle$ form a linearly independent set. Before delving into this topic, we investigate known obstructions to a full rank solution, and the role of resonant eigenvalues.  

\section{Resonance}
\label{Sec: R}
 
Following the terminology of the literature, we say $\lambda_j$ is a ``resonant eigenvalue" of $\phi'(0)$ if (and only if) $\lambda_j=\lambda_1^{k_1} ... \lambda_m^{k_m}$ for some non-negative integers $k_i$ satisfying $\sum k_i>1$ (for example \cite[p.2]{Enoch}). If we assume $\phi'(0)$ is upper triangular (or better in Jordan form) so that $\C$ is lower triangular, then we notice that $\lambda_j$ appears on the diagonal of $\C$ below $\phi'(0)^t$ if and only if it is a resonant eigenvalue.  In \cite{CM03}, Cowen and MacCluer identify that such resonance can cause an obstruction to finding a full rank solution.  On the other hand, examples of a full-rank solution in the presence of resonance are given \cite[Examples 2,3]{CM03}. 

	Subsequently, \cite{Enoch} addresses Schroeder's equation in several variables from a completely algebraic point of view, representing all functions as formal power series.  In other words, the coordinate functions of $\phi$ and $F$ are considered as column vectors (with the same notation as the current paper) but with no consideration of convergence of the corresponding Taylor series.   
	
\begin{theorem}  \cite[Theorem 4.3]{Enoch}
\label{enoch's theorem}
Let $\phi$ be an analytic self-map of the ball, fixing 0, $\phi'(0)$ non-singular, and $\phi$ not unitary on any slice.  If $\phi$ has no resonant eigenvalues, then there exists a formal power series solution to Schroeder's equation with full rank near 0.   
\end{theorem}
To understand this fully, write $F=(f_1, ..., f_n)^t$, and then write
$$f_j=\left[\begin{matrix}x_1^j \\ x_2^j \\ \vdots \end{matrix} \right].$$
Now, writing $F\circ \phi=\phi'(0)F$ and $F'(0)=I_n$  gives an infinite system of equations with the $x_i^j$ as the unknowns.  In the absence of resonance, Enoch has shown that the $x_i^j$ can be found to satisfy these equations.   The issue which remains is of course whether the corresponding  power series for each $f_j$ actually converges on $\Bn$.  

\begin{theorem}
\label{formal p.s. theorem}
Let $\phi$ be an analytic self-map of the ball, fixing 0, $\phi'(0)$ non-singular, and $\phi$ not unitary on any slice.  Any formal power series $F$ satisfying $\C F=\phi'(0)F$ is indeed analytic on $\Bn$.  
\end{theorem}	
This theorem follow directly from the uniqueness statement of Theorem \ref{S: three equiv. Theorem}.  In particular, Corollary \ref{A:l1} and the following paragraphs facilitate this line of reasoning.   When $\phi'(0)$ is upper triangular and $\C$ lower triangular (which can always be assumed with no loss of generality), the infinite system of equations written at once by Schroeder's equation can be solved inductively.  In fact, Theorem \ref{S: main theorem} assures us that at least one solution exists, although there may be many as in \cite[\S Example 2]{CM03}, and Theorem \ref{formal p.s. theorem} assures us all solutions found in this purely algebraic manner are indeed analytic.  In light of this result, Enoch's algebraic method may be the best way to find a solution in practice.    
 
\begin{corollary}
\label{absence of resonance}
Let $\phi$ be an analytic self-map of the ball, fixing 0, $\phi'(0)$ non-singular, and $\phi$ not unitary on any slice.  If $\phi$ has no resonant eigenvalues, then a full rank solution to Schroeder's equation exists.  
\end{corollary}
	We have obtained that resonance is necessary whenever there is no full rank solution, but what is obscure is why resonance is not sufficient to prevent a solution (recall that \cite{CM03} provides examples of a full rank solution in the presence of resonance.)   To illustrate this phenomenon, we provide an example similar to \ref{S: ex 0} which has a full rank solution.
 
\begin{example}
\label{R: ex2}
Now set $\phi(z_1,z_2)=(1/2)(z_1, z_2/2)$.  Similar to the \ref{S: ex 0}, it is easy to see that $\phi$ satisfies our hypotheses with domain $\mathbb{B}^2$ and experiences the same resonance.  We now have  
$$\C=\begin{array}{cc} 
& \begin{matrix}
\phantom{..}\phi_1 \phantom{.}& \phi_2\phantom{,} & \phi_1^2 \phantom{.} & \phi_1\phi_2  \phantom{.}\dots
\end{matrix}\\ 
\begin{matrix}
z_1\\ z_2\\ z_1^2 \\ z_1z_2\\ \vdots
\end{matrix}
&\left[\begin{matrix}
1/2 & 0 	& 0 	&0& \dots\\
0 	& 1/4	& 0 	&0& \dots \\
0 	& 0		& 1/4  	& 0&\dots \\
0 	& 0 	& 0 	& 1/8 & \dots\\
\vdots&\vdots&\vdots&\vdots& \ddots
\end{matrix} \right]
\end{array}=\mbox{diag}(\frac{1}{2},\frac{1}{4},\frac{1}{4},\frac{1}{8},...)$$
so it is easy to see that $f_1=z_1, $ and $f_2=z_2$ are eigenfunctions with eigenvalues 1/2 and 1/4, respectively.  Thus $F=(af_1,bf_2)^t$  with any nonzero $a$ and $b$ gives a full rank solution.  
\end{example}

In both examples $\p$ has one (1/4)-Jordan block of length 1, but  $\C$ has one (1/4)-Jordan block of length 2 in Example \ref{S: ex 0} and two (1/4)-Jordan blocks of length 1 in Example \ref{R: ex2}.  It is the fact that the original (1/4)-Jordan block of $\p$ gained length in the first example that forces any (1/4)-eigenfunction to have both first derivatives 0, thereby preventing a full rank solution.   Resonance is only responsible for the appearance of 1/4 on the diagonal of $\C$ below $\p$.  

We shall see that the only possible obstruction to a full rank solution is illustrated by the Example \ref{S: ex 0}.  Roughly speaking, what goes wrong is that if an original Jordan block of $\p$ gains length when $U$ is put into Jordan form, then it is necessary for the corresponding eigenfunction to have all of its first derivatives 0.  Since a solution $F$ must include an eigenfunction for each original Jordan block of $\p$, any such solution $F$ cannot be full rank.  

\section{Criteria for a Full Rank Solution}
\label{Sec: CfaFRS}
By contrast to this paper, \cite{CM03} assumes with no loss of generality that $\phi'(0)$ is upper triangular by conjugating $\phi$ with a unitary matrix.  Since a unitary change of basis preserves the ball and even the norm in $\A$, few of the concerns of  Section \ref{Sec:SA} arise.  The main theorem of \cite{CM03} is that if $\phi'(0)$ is diagonalizable (in addition to our hypotheses), then a full rank Schroeder solution exists if and only if $\dim[\ker(\C-\mu I)]$ is the number of times $\mu$ appears on the diagonal of $\C$.  Equivalently, that $U$ is also diagonalizable precisely when a full rank solution exists.  Our main theorem, (Theorem \ref{C: main theorem part 2}) is an extension of this result to the general case.   

While in Section \ref{Sec. Sw/oCRN0} it was useful to decompose $\Cn$ into the Jordan subspaces, in the current section it will be useful to consider generalized eigenspaces.  Recall that if $L$ is a linear operator on a vector space $X$, for each distinct eigenvalue, $\mu,$ the corresponding generalized eigenspace is 
$$T_{\mu}= \bigcup_{k=1}^{\dim(X)} \ker[(L-\mu I)^k].$$
In our case, $L=\phi'(0)^t$ with  indistinct eigenvalues $\lambda_1, ..., \lambda_m$.  Thus we let $\mu_1,..., \mu_k$ be the distinct eigenvalues of $\phi'(0)^t$, and denote the generalized eigenspace decomposition   $\Cn=T_{\mu_1}\oplus ... \oplus T_{\mu_k}.$  

We will denote the compression of $\phi'(0)^t$ to $T_{\mu_j}$ by $L_j$, so that in block form $\phi'(0)^t= $ diag$(L_1, ... , L_k)$, and each $L_j$ is just the Jordan form matrix with those Jordan blocks of $ \phi'(0)^t$ that have $\mu_j$ on their diagonal.  Correspondingly, we write $F=(F_{\mu_j})^t, $ in block form.  For example, if we assume $\lambda_1, ..., \lambda_s$ are the only eigenvalues equal to $\mu_1$, we have $L_1=$diag$(J_1, ... , J_{s})$,  and $F_{\mu_1}=(F_1, ..., F_{s})^t$ with $F_j$ as defined in Section \ref{Sec. Sw/oCRN0}. Now, $F\circ\phi=\phi'(0)F$ becomes the system of equations 
\begin{equation}
\label{C: e1}
F_{\mu_j}\circ\phi=L_j^tF_{\mu_j} 
\end{equation}
for $1\leq j \leq k.$

The advantage of reorganizing into generalized eigenspaces is that  we can reduce the problem to finding a full rank solution $F_{\mu_j}$ for each $j$.  Specifically, if each $F_{\mu_j}$ can be found to satisfy (\ref{C: e1}) and have full rank near 0, then Proposition \ref{C: p1} shows that $F$ is indeed a full rank solution.  

\begin{proposition}
\label{C: p1}
If Equation (\ref{C: e1}) is satisfied and each $F_{\mu_j}$ has linearly independent component functions, then $F$ has  linearly independent component functions. In particular, $F$ is a full rank solution if and only if Equation (\ref{C: e1}) holds with each $F_{\mu_j}$ full rank near 0.  
\end{proposition}
\begin{proof}
Since each coordinate function of $F_{\mu_j}$ is a member of the generalized eigenspace of $\C$ with eigenvalue $\mu_j$ and the $\mu_j$ are distinct, the component functions of $F$ are linearly independent as soon as the component functions of each $F_{\mu_j}$ are.  
\end{proof}

Thus we fix an eigenvector $\mu$ and seek a full rank solution to $F_{\mu}\circ\phi=L_j^tF_{\mu}$.

\begin{theorem}
\label{C: mt}
\label{C: main theorem 1}
Let $\phi$ be an analytic map on the ellipsoid $\mathcal{D}$ fixing 0 with $\phi'(0)$ a full rank upper-triangular Jordan form matrix. Suppose also that $\C$ is compact on the Bergman space, $\h=A_H^2(\D)$ as defined above.
A full rank solution $F_\mu$ exists if and only if $Q(\ker(U-\mu I_N))=\ker(\p-\mu I_n)$, with $Q$ the orthogonal projection to $\langle z_1, ... , z_n\rangle.$  Consequently, 
a full rank solution $F$ exists if and only if 
$Q(\ker(U-\mu_j I_N))=\ker(\p-\mu_j I_n)$ for every eigenvalue $\mu_j$ of $\p$.   
\end{theorem}
\begin{proof}
  With no loss of generality, let $\mu= \lambda_1= ... = \lambda_s$ and assume $\lambda_j\neq \mu $ for any $s<j\leq m$.  Thus, the generalized eigenspace of $\p$ with eigenvalue $\mu$ is $ T:=T_{\mu}=S_1\oplus ... \oplus S_{s},$ and  as before we have the chain $e_1^j, ... , e_{n_j}^j$ forming a basis of $S_j$  with $n_j=\dim(S_j)$.  
Recall that since $Q$ is the projection from $\h$ to $\langle z_1, ... , z_n\rangle$,   $Qg=\nabla g(0)$ for any $g$ in $\h$.  So, 
$$F_\mu=[F_1, ..., F_{m_0}]^t=[(f_1^1, ... , f_{n_1}^1), (f_2^2, ... , f_{n_2}^2), ... , (f_1^s, ... , f_{n_s}^{s})]^t$$ 
has full rank near 0 if and only if $\{Qf_i^j : 1\leq j \leq s, 1\leq i \leq n_j \}$ is linearly independent.  

Suppose $Q(\ker(U-\mu I_N))=\ker(\p-\mu I_n)$.  There exist Jordan decompositions $\Cn=\oplus_{j+1}^nS_j$ and $\CN=\oplus_{j=1}^l\tilde S_j$ for $\p$ and $U$ respectively, satisfying the result of Theorem \ref{LA: workhorse theorem}. That is, for each $1\leq  j \leq s$ the chains $e_1^j, ..., e_{n_j}^j$ and $\E_1^j, ..., \E_{n_j}^j$ are bases for $S_j$ and $\tilde S_j$ respectively, and they satisfy $Q_j\E_i^j=e_i^j$.  Now Theorem \ref{S: three equiv. Theorem} gives unique $F_j=(f_1^j, ... , f_{n_j}^j)^t$ satisfying both $F_j\circ \phi=J_j^tF_j$ and $P(f_i^j)=\E_{i}^j$.  Since $Q_j(f_i^j)=Q_j(\E_i^j)=e_i^j$, we have $Q(f_i^j)\neq 0$.  It follows that $Q(f_1^j), ... , Q(f_{n_j}^j)$ is a chain of $\p$ for each $j$.  Lemma \ref{LA: chain lin. ind lemma} shows that  $\{Qf_i^j : 1\leq j \leq s, 1\leq i \leq n_j \}$ is a linearly independent set, which is equivalent to $F_\mu$ having full rank.  

Conversely,  suppose that $F_\mu$ is a full rank solution.  This implies that $\ker(\p-\mu I_n)=\langle Qf_{n_j}^j : 1\leq j \leq s \rangle$.  Since $Pf_{n_j}^j$ with $1\leq j \leq s$ are eigenvectors of $U$, we have the result.  
\end{proof}
Section \ref{Sec:SA} now allows us to formulate the result under the general hypotheses, that is, for $\phi$  an analytic self-map of $\Bn$, fixing 0, not unitary on any slice, and with full rank near $0$.  The theorem above applies to $D\phi D^{-1}$, and 
$$C_{D\phi D^{-1}}=\left[\begin{matrix}
U&0\\V&W
\end{matrix}\right],$$
 where $D$ is chosen so that $(D\phi D^{-1})'(0)$ is an upper triangular Jordan form matrix.  So, to state the theorem in terms of $\phi$ we just need to translate ``$Q[(\ker(U-\mu I_N)]=\ker[(D\phi'(0)^tD^{-1})^t-\mu I_n]$" to the language of $\C.$   Recall that $\C$ is compact on $A_G^2(\Bn)$, and $C_{D\phi D^{-1}}$ is compact on $A_H^2(D\Bn)$, where $\iota(f)=f\circ D^{-1}$ is an isometric isomorphism from $A_G^2(\Bn)$ to $A_H^2(D\Bn)$.  Write
$$C_{\phi}=\left[\begin{matrix}
\tilde U&0\\ \tilde V&\tilde W
\end{matrix}\right],$$ 
in block form acting on 
 $A_G^2(\Bn)=\tilde H_1\oplus \tilde H_2$, with $\tilde H_1=\langle z^\alpha\in A_G^2(\Bn) : |\alpha|\leq K\rangle$ and $K$ large enough that $\tilde W $ shares no eigenvalue with $\p$.  (It should be noted that the general $\C$ is block lower triangular, as evidenced above.  To convince oneself of this, notice that $\phi^\alpha$ cannot have a nonzero derivative of order less than $|\alpha|$.)  Since $\iota $ preserves homogeneous polynomials, we see that $\iota (\tilde H_1)=H_1,$ and  $\iota (\tilde H_2)=H_2$.  Consequently, $\tilde U=\iota^{-1} U\iota$, and $Q(\ker(U-\mu I_N))=\ker((D\phi'(0)^tD^{-1})^t-\mu I_n)$ is equivalent to $\tilde Q(\ker(\tilde U -\mu I_N))=\ker(\p-\mu I_n)$, where $\tilde Q$ is the orthogonal projection from $A_G^2(\Bn)$ to $\langle z_1, ... , z_n\rangle$.  Thus a full rank solution, $F_\mu$ exists if and only if  $\tilde Q(\ker(\tilde U -\mu I_N))=\ker(\p-\mu I_n)$.  Since in general, 
 $$U=\left[ \begin{matrix} \phi'(0)^t & 0 \\ Y & Z \end{matrix}\right],$$
  $\tilde Q(\ker(\tilde U -\mu I_N))\subseteq \ker(\p-\mu I_n)$, so we can just count the dimensions to see if equality holds.   
\begin{theorem}[Main Theorem]
\label{C: main theorem part 2}
Let $\phi$ be an analytic self-map of $\Bn$, fixing 0, not unitary on a slice, and with $\phi'(0)$ full rank. We know that $\C$ is compact on the Bergman space $A_G^2(\Bn),$ for appropriate $G$.  Let $\tilde Q$ be the orthogonal projection of  $A_G^2(\Bn)$ to $\langle z_1, ... , z_n\rangle$.  Fix $\mu \in \sigma(\phi'(0)),$ the spectrum of $\phi'(0)$, and let 
$$C=\left[ \begin{matrix} \tilde U & 0 \\ \tilde V & \tilde W \end{matrix} \right] $$
with $\tilde U$ the compression of $\C$ to $\langle z^\alpha: |\alpha|\leq K\rangle $, where $K$ is large enough that $\mu \notin \sigma(\tilde W).$   Now, there is an analytic $F_\mu$ with full rank near 0 satisfying Equation (\ref{C: e1}) if and only if $\tilde Q(\ker(\tilde U -\mu I_N))=\ker(\p-\mu I_n)$, if and only if  $\dim[\tilde Q(\ker(\tilde U -\mu I_N))]=\dim[\ker(\p-\mu I_n)]$.  Finally, there is a full rank solution $F$ satisfying $F\circ \phi=\phi'(0) F$ if and only if $\dim[\tilde Q(\ker(\tilde U -\mu I_N))]=\dim[\ker(\p-\mu I_n)]$ for each eigenvalue $\mu$ of $\phi'(0)$.
\end{theorem} 

As is the practice in \cite{CM03}, it is possible to state this main result without appealing to the underlying Hilbert space or even compact operators, and perhaps a reader with less interest in the operator theory involved and more interest on the functionality of the theorem will appreciate the formulation provided below. Furthermore, we point the interested reader to compare and contrast Theorem \ref{main theorem revisited} with \cite[Theorem 14]{CM03}. In both theorems it is possibly easiest to consider $\C$ only as a large (infinite) matrix. 

\begin{theorem}[Main Theorem Revisited]
\label{main theorem revisited}
Let $\phi$ be an analytic self-map of $\Bn$, fixing 0, not unitary on a slice, and with $\phi'(0)$ full rank.  Write $\sigma(\phi'(0))=\{\lambda_1, ..., \lambda_m\}$. 
Let $N=\max\{\sum k_i: \lambda_1^{k_1}  \cdot \cdot \cdot \lambda_m^{k_m}\in \sigma(\phi'(0)) \}$  and write 
$$\C=\left[\begin{matrix} U & 0 \\ V & W \end{matrix} \right]$$ 
with $U$ the upper left $N\times N$ corner of $\C$.   Let $Q:\CN \to \Cn$ be the projection to the first $n$ components.   A full rank solution to Schroeder's equation exists if and only if $\dim[\ker(Q(U-\lambda_j I_N))]=\dim[\ker(\phi'(0)^t-\lambda_j I_n)]$ for every $j=1,...,m.$  
\end{theorem}
\begin{proof}
It suffices to show that $\sigma(\phi'(0))\cap \sigma(W)=\emptyset$, as the rest follows from \ref{C: main theorem part 2}.
Choose $D$ to be a unitary $n\times n$ matrix so that $D\phi'(0)D^{-1}$ is upper triangular.  
Note that $D\Bn=\Bn$  since $D$ is a unitary.  
We know $C_{D\phi D^{-1}}$ is lower triangular and compact on $A_G^2(\Bn)$, and that $C_D$ is a unitary on $A_G^2(\Bn)$.   
Now by the choice of $N$ we see that no $\lambda_j$ can appear on the diagonal of $C_{D\phi D^{-1}}$ below the $N^{th}$ row.  
Since $C_{D\phi D^{-1}}=C_{D}^{-1}\C C_{D}$ it follows that $\sigma(\phi'(0))\cap \sigma(W)=\emptyset$.
\end{proof}

While Jordan form of both $\p$ and $U$ played a large role in understanding and proving the our main results, none of the statements of these results rely on any special matrix form.  The main theorem of the section does not rely on any Jordan form, but is instead formulated in terms of invariant subspaces of our operator.  This is important for practicality of the theorem; that is, given a specific $\phi$, one can theoretically obtain the result of the theorem without first putting $\p$ into Jordan form.  

The following are a few results that can expedite our understanding of when a full rank solutions exists. For the remainder of the section, let $\C$ and $\phi$ be as in the hypotheses of Theorem \ref{C: main theorem part 2}.  Theorem \ref{C: t1} is a small generalization of Theorem \ref{absence of resonance}.     
\begin{theorem}
\label{C: t1}
If $\mu$ is not a resonant eigenvalue, then there exists an $F_{\mu}$ satisfying Equation (\ref{C: e1}) with full rank near 0.  
\end{theorem}
\begin{proof}
Without loss of generality, we assume $\phi'(0)$ is an upper triangular Jordan form matrix.  
Since  $\mu$ does not occur on the the diagonal of $\C$ below $\p$, it follows from the proof of Theorem \ref{LA: J block theorem} that no Jordan block of $\p$, $J_j$, can gain length on any induction step of the proof.  By Theorem \ref{LA: workhorse theorem} we see $Q(\ker(U-\mu I_N))=Q(\p-\mu I_n)$, so the result follows from our main theorem.  
\end{proof}

\begin{theorem}
\label{C: counting argument}
If $\mu$ is a resonant eigenvalue, and $\dim(\ker(U-\mu I_N))=\dim(\ker(\p-\mu I_n))$, then there is not a full rank solution $F_\mu$.   
\end{theorem}
The suggestion of using a counting argument was proposed by Carl C. Cowen, for which the author is thankful.  
\begin{proof}
Without loss of generality, we assume $\phi'(0)$ is an upper triangular Jordan form matrix.  
Since $\mu$ appears on the diagonal of $U$ below  $\p$ and $\dim(\ker(U-\mu I_N))=\dim(\ker(\p-\mu I_n))$, it follows from the proof of Theorem \ref{LA: J block theorem} that an original Jordan block gains length at some step of the induction.  By Theorem \ref{LA: workhorse theorem} we see $Q(\ker(U-\mu I_N))\subsetneq\ker(\p-\mu I_n)$, and the result follows from Theorem \ref{C: main theorem 1}.
\end{proof}


\begin{example}
\label{C: e 1}
Let $\phi(z)=(\frac{z_1}{2}, \frac{z_2}{4}+\frac{z_3}{8}+\frac{z_1^2}{8}, \frac{z_3}{4}, \frac{z_4}{8})$ on $\mathbb{B}^4$.   Clearly $|\phi(z)|^2 \leq \frac{|z_1|^2}{4}+ \frac{9|z_2|^2}{16}+ \frac{9|z_3|^2}{64}+ \frac{9|z_1|^4}{64} +\frac{|z_3|^2}{16}+\frac{|z_4|^2}{64}\leq \frac{41|z|^2}{64}$ so $\phi$ is not unitary on any slice, and fixes 0.   Now we have
$$\p=\left[ \begin{matrix} 
1/2 &0&0&0\\
0& 1/4&0&0 \\
0& 1/8 &1/4&0\\
0&0&0&1/8
\end{matrix}\right]$$ 
and 
$$\C=\begin{array}{cc} 
& \begin{matrix}
\phi_1 & \phantom{,} \phi_2 \phantom{,} & \phi_3 \phantom{.,} & \phi_4 \phantom{,}&\phantom{.} \phi_1^2 & \phi_1\phi_2  & \phi_1\phi_3 &\dots
\end{matrix}\\ 
\begin{matrix}
z_1\\ z_2\\ z_3\\z_4\\z_1^2 \\ z_1z_2\\z_1z_3 \\ \vdots
\end{matrix}
&\left[\begin{matrix}
1/2 & 0 	& 0		&0		&0		&0		&0&\dots\\
0 	& 1/4	&0	 	&0		&0		&0		&0&\dots \\
0 	& 1/8	& 1/4  	& 	0	&0		&0		&0&\dots	\\
0 	&0		&0	 	& 1/8	&  0	&0		&0&\dots	\\
0	&1/8	&0		&0		&	1/4	&	0	&0&\dots\\
0	&0		&0		&0		&0		&	1/8	&0&\dots\\
0	&0		&0		&0		&0		& 1/16		&1/16&\dots\\
\vdots&\vdots&\vdots& \vdots&\vdots	&\vdots	 &\vdots	&\ddots
\end{matrix} \right].
\end{array}$$

By inspection we see that $F_{1/2}=z_1$, and $F_{1/8}=z_4$.  So the existence of a full rank solution rests on having a full rank $F_{1/4}$.  $U$  can be chosen to be the upper left $6 \times 6$, as all diagonal values below this are less than or equal to 1/16.  Since $\ker(\p-I_4/4)=\{e_3\}\subset \{e_3, e_5\}=\ker(U-I_6/4 )$, we know a full rank solution exists.  Now, to find this solution notice that when in Jordan form, $\p$  has only one (1/4)-Jordan block of length 2, and $e_2, e_3/8$ gives the chain basis of this block. This implies that $F_{1/4}=(f_1, f_2)$ with $f_1, f_2$ a chain of $\C$ that is determined by the corresponding chain of $U$ (see Theorem \ref{S: three equiv. Theorem}).  Since $(U-I_6/4)e_2=e_3/8+e_5/8$, we have $F_{1/4}=(z_2, z_3/8+z_1^2/8)^t$.  In total $F=(z_1, z_2, z_3/8+z_1^2/8, z_4).$

\end{example}

\section{Solutions to $C_\phi F = \phi(0)^k F$}
\label{Section: generalized solutions}
It is well known in one variable that if $\phi:\mathbb{D}\to\mathbb{D}$ with $\phi(0)=0$ and $0<|\phi'(0)|<1$ then $\phi'(0)^k$ are eigenvectors of $\C$ for each non-negative integer $k$.  Moreover, given $f\neq0$ so that $\C f=\phi'(0)f$, it is easy to see that $\C f^k=\phi'(0)^k f$ for $k\geq 1$.  Whether or not this phenomenon generalizes to several variables is evidently a natural question to ask in the current circumstances, as the author was asked such a question by \^{Z}eljko \^{C}u\^{c}kovi\'{c} after a presentation at the 27th Southeast Analysis Meetings, University of Florida, Gainesville, March 17-19, 2011.  The author thanks Professor \^{C}u\^{c}kovi\'{c} for the insightful inquiry, and has included the current section as a belated answer.    
\begin{lemma}
\label{ $J^k$ lemma}
Let $J$ be a lower triangular $\lambda$-Jordan block with basis vectors $e_1,... , e_s$.  Then for any positive integer $k$
 there is a basis $\E_1, ... , \E_s$ so that $\E_s=e_s$, and when expressed in this basis $J^k$ is a lower triangular $\lambda^k$-Jordan block.   
\end{lemma}
\begin{proof}
Working in the original basis, $\{e_1, ... , e_s\}$ $J^k$ is the lower triangular matrix, 

$$J^k=\begin{matrix}
e_1 \\ e_2 \\ \vdots \\ e_s
\end{matrix}\left[ \begin{matrix}
\lambda^k   \\
k\lambda	& \lambda^k \\
 	*		& \ddots 	& \ddots \\
	*		&		*	& k\lambda	& \lambda^k
 \end{matrix}\right].$$
Seeing our result is trivial for $s=1$, we induct on $s$.  Applying our induction hypothesis to the upper left $(s-1) \times (s-1)$ we are furnished with a new basis, namely $\E_1, ... , \E_{s-2}, e_{s-1}, e_s$, and we have  
$$J^k=\begin{matrix}
\E_1 \\ \E_2 \\ \vdots \\ e_{s-1} \\ e_s
\end{matrix}\left[ \begin{matrix}
\lambda^k   \\
1	& \lambda^k \\
 	& \ddots 	& \ddots \\
	&			& 1			& \lambda^k\\
a_1	& \dots 	& a_{s-2} 	& k\lambda & \lambda^k 		
 \end{matrix}\right].$$
 An application of Lemma \ref{LA: one J block lemma} gives our result.  
\end{proof}

\begin{theorem}[Main theorem of Section \ref{Section: generalized solutions}]
\label{main generalized theorem}
Let $\phi$ be an analytic self-map of $\Bn$ so that $\phi(0)=0$, $\phi'(0)$ has full rank, and $\phi$ is not unitary on any slice.  Given a positive integer $k$, there is an analytic $F:\Bn \to \Cn$ with linearly independent component functions satisfying 
$$\C F = \phi'(0)^k F.$$
Furthermore, if $k>1$ then no such $F$ can have full rank near 0.    
\end{theorem}
\begin{proof}
To realize the Jordan form of $\phi'(0)$, let $D$ be a nonsingular matrix so that $D\phi'(0)D^{-1}=$ diag$(J_1^t, ... , J_m^t)$ with each $J_j$ a lower triangular Jordan block (so $J_j^t$ are the upper triangular counterparts).  Say $J_j$ has eigenvalue $\lambda_j$ and length $n_j$ as usual.  
Similarly,  let $E$ be a nonsingular matrix so that $E\phi'(0)^kE^{-1}$ is an upper triangular Jordan form matrix. It follows from Lemma \ref{ $J^k$ lemma} that $E\phi'(0)^kE^{-1}$ is identical to $D\phi'(0)D^{-1} $ except the diagonal entries, $\lambda_j$,  are replaced with $\lambda_j^k$.  

Since  $(E\phi'(0)E^{-1})^k=E\phi'(0)^kE^{-1}$, $C_{E\phi E^{-1}} F= (E\phi'(0)E^{-1})^k F$ if and only if $\C (E^{-1}FE)=\phi'(0)^k (E^{-1}FE)$. Ergo, it suffices to find an $F$ which satisfies $C_{E\phi E^{-1}} F= (E\phi'(0)E^{-1})^k F$, and has linearly independent component functions.  Mimicking Lemma \ref{S: three equiv. lemma}, we see that such an $F$ exists if and only if for each $1\leq j\leq m$, $\C$ has a chain of length $n_j$ and eigenvalue $\lambda_j^k$ and that the chains form a linearly independent set.   

   Set $\psi=D\phi D^{-1}$.  It follows from the proof of Proposition \ref{SA: p1} that $\C=\iota^{-1}C_{\psi}\iota$ with $\iota$ an isometric-isomorphism.  Consequently, it suffices to show that the corresponding chains exist for $C_{\psi}$.  
  
   For notational simplicity, let $\lambda_1=\lambda$, and set $n_1=s$.  Now, since $\psi'(0)$ is in Jordan form, we see $\psi_j(z)=\lambda z_j + z_{j+1}+ O(2)$ for  $1\leq j<s$, and $\psi_s(z)=\lambda z_s +O(2)$.  We calculate, 
\begin{align*}
\psi_1^k&=[\lambda z_1 +z_2 +O(2)]^k\\
		&=\lambda^kz_1^k + k \lambda^{k-1}z_1^{k-1} z_2 +   \dots +   z_2^k + O(k+1) \\
\psi_1^{k-1}\psi_j&= [\lambda z_1 +z_2 +O(2)]^{k-1}[\lambda z_j+z_{j+1}+O(2)]\\
				  &= [\lambda^{k-1}z_1^{k-1} + (k-1) \lambda^{k-2}z_1^{k-2} z_2 + \dots + \\
				  &  z_2^{k-1} +O(k)][\lambda z_j+z_{j+1}+ O(2)] \\
				  &=\lambda^kz_1^{k-1}z_j+\lambda^{k-1}z_1^{k-1}z_{j+1}+ ... +z_2^{k-1}z_{j+1}+ O(k+1).
\end{align*}
 Hence,  $C_{\psi}$ contains the following sub-matrix along it's diagonal, 
$$  \begin{array}{cc} \phantom{a} &\phantom{[}\begin{matrix}  & \phi_1^k	& \phi_1^{k-1}\phi_2& \dots & \phi_1^{k-1}\phi_s \end{matrix}\phantom{]} \\
\begin{matrix}
z_1^k\\
z_1^{k-1}z_2\\
\vdots \\
z_1^{k-1}z_s
\end{matrix}&
					\left[ \begin{matrix}
					\phantom{0}\lambda^k \phantom{0} & \phantom{0} 0 \phantom{.}	& \phantom{0} 0 \phantom{0}	& 0\phantom{00}\\
					\phantom{0} \lambda^{k-1}\phantom{0}	& \phantom{0}\lambda^k	\phantom{.}	&\phantom{0} 0 \phantom{0}	& 0\phantom{00} \\
					\phantom{0}0\phantom{0}& \ddots					& \ddots			& 0\phantom{00}\\
					\phantom{0}0\phantom{0}	& \phantom{0}0\phantom{.}	&\phantom{0}\lambda^{k-1}\phantom{0}	& \lambda_1^k \phantom{00}  \\
\end{matrix} \right]
\end{array}.$$
Similarly, we have such a sub-matrix of $C_{\psi}$ for each $\lambda_j$ along the diagonal of $C_\psi$.  It follows from Theorem \ref{S: U, C chain correspondence} and Theorem \ref{LA: extra theorem}  that for each $1\leq j \leq m$, $C_\psi$ has a chain of length $n_j$ and eigenvalue $\lambda_j$, and that these chains are linearly independent.  

We have proven that $\C F=\phi'(0)^k F$ can be solved with an $F$ having linearly independent component functions.  It remains to show that any such $F$ must have $F'(0)$ singular if $k\geq 2$.  Let $k\geq 2$, $A=F'(0),$ and $B=\phi'(0)$.  Using the notation above we know the spectrum of $B$ is $\sigma(B)=\{\lambda_1, ... , \lambda_m\}$ and $\sigma(B^k)=\{ \lambda_1^k, ... , \lambda_m^k\}.$  With no loss of generality,  set $\lambda =\lambda_1,$ and assume $|\lambda|=\max \sigma(B)$. Differentiating $F\circ \phi=\phi'(0)^k F$ and evaluating at $z=0$ gives $AB=B^kA$. Hence,if we let $e$  be a nonzero member of $\ker(B-\lambda I_n)$ we obtain 
$$ B^k(Ae)=ABe=\lambda Ae.$$ 
Since $k\geq 2$,  $|\lambda| >|\lambda|^k=\max \sigma(B^k)$.   It follows that $\lambda $ is not within $\sigma(B)$, and thus, $Ae=0$, which concludes the proof.   
\end{proof}
The existence proof above gives little insight  on how to find such a solution in practice.  Driven by the single variable setting, in which $\C f= \phi'(0) f$ implies $\C f^k=\phi'(0)^k$ for any positive integer $k$, we have the following useful result.  
\begin{theorem}
\label{How to find F }
Suppose $F=(f_1, ... , f_s)^t$ has linearly independent component functions and satisfies $\C F= J^t F$ with 
$$J^t=\left[\begin{matrix}
\lambda & 1 \\
		& \ddots 	& \ddots \\
		&			&\lambda	& 1\\
		&			&			&\lambda
 \end{matrix}\right]$$ 
 an upper triangular $s\times s$ Jordan block, and $\lambda \neq 0$.      
 For an integer $k\geq 1$, let $E$ be the $s\times s$ invertible matrix provided by Lemma \ref{ $J^k$ lemma} so that 
 $$E (J^t)^k E^{-1}=\left[\begin{matrix}
\lambda^k & 1 \\
		& \ddots 	& \ddots \\
		&			&\lambda^k	& 1\\
		&			&			&\lambda^k
 \end{matrix}\right].$$    
 Then $$G=E^{-1}\left(\frac{f_1 f_s^{k-1}}{\lambda^{(k-1)(s-1)}}, ... , \frac{f_{s-1} f_s^{k-1}}{\lambda^{(k-1)(1)}}, \frac{f_s^k}{\lambda^{(k-1)(0)}}  \right)^t$$ satisfies $\C G= (J^t)^k G$ and has linearly independent component functions. 
\end{theorem}  
\begin{proof}
$\C F=J^t F$ is equivalently written as 
$$(\C-\lambda I): f_1 \mapsto f_2 \mapsto ... \mapsto f_s \mapsto 0.$$  
For $j=1, ... , s$ set
$$h_j= \frac{f_j f_s^{k-1}}{\lambda^{(k-1)(s-j)}},$$
and $H=(h_1, ... , h_s)^t.  $
 A simple calculation yields
$$\C h_s=\C f_s^{k} = \lambda^k f_s^{k}=\lambda^k h_s,$$
and for $1\leq j <s$
\begin{align*}
\C h_j 	&= \C \frac{f_j f_s^{k-1}}{\lambda^{(k-1)(s-j)}}\\
		&= \frac{(\lambda f_j + f_{j+1})\lambda^{k-1}f_s^{(k-1)} }{\lambda^{(k-1)(s-j)}}\\
		&=\lambda^k \frac{f_j f_s^{(k-1)} }{\lambda^{(k-1)(s-j)}}+\frac{f_{j+1}f_s^{(k-1)}}{\lambda^{(k-1)(s-j-1)}}\\
		&=\lambda^k h_j + h_{j+1}.
\end{align*}
Equivalently, 
$$(\C-\lambda^k I): h_1\mapsto h_2 \mapsto ... \mapsto h_s \mapsto 0.$$
So we have $$\C H= \left[\begin{matrix}
\lambda^k & 1 \\
		& \ddots 	& \ddots \\
		&			&\lambda^k	& 1\\
		&			&			&\lambda^k
 \end{matrix}\right] H = (E (J^t)^k E^{-1})H.$$
 It follows that $G:=E^{-1}H$ satisfies $\C G=(J^t)^k G$ as desired.  

Lastly, note that $\{h_j\}_j=f_s^{(k-1)}\{c_1f_1, ..., c_sf_s \}$ with $c_j=\lambda^{-(k-1)(s-j)}$, and so is a linearly independent set.   
Now for any non-zero $a\in \mathbb{C}^s$ we  have $  0\neq \langle H, a\rangle =\langle EG,a\rangle =\langle G, E^* a \rangle.$  Since $E^*$ is bijective, this implies that $G$ has linearly independent component functions. 
\end{proof}
Note that if $\phi'(0)$ has full rank, then none of it's eigenvalues can be zero, so the hypothesis $\lambda\neq 0$ in the theorem above is not restrictive.  Sections \ref{Sec:SA} and \ref{Sec. Sw/oCRN0} assure us we can  find a full rank solution to Schroeder's equation, and allow us to first put $\phi'(0)$ into Jordan form, and then find such a solution (by finding the chains of $U$ and subsequently $\C$). Thus, Theorem \ref{How to find F } now allows us to explicitly find solutions to $\C F= \phi'(0)^k F$ that have linearly independent component functions.  If $k>1$, Theorem \ref{main generalized theorem} says that no full rank solution exists, so in some sense this is the best we can do.  The case $k=1$ of course is handled in Section \ref{Sec: CfaFRS}.  

The following corollary is known,  for example \cite[Theorem 7.20]{CMbook}, but since it follows immediately, we state it for completeness.  

\begin{corollary}
Let $\phi$ be a self-map of $\mathbb{B}^n$ fixing 0, and suppose $\phi$ is not unitary on any slice, and $\phi'(0)$ has full rank.  Then for all $\lambda \in \sigma(\phi'(0))$, and for all non-negative integers $k$, $\lambda^k$ is an eigenvalue of $\C$.  
\end{corollary}   

\bibliographystyle{amsalpha}
\bibliography{SEbib}

\providecommand{\bysame}{\leavevmode\hbox to3em{\hrulefill}\thinspace}
\providecommand{\MR}{\relax\ifhmode\unskip\space\fi MR }
\providecommand{\MRhref}[2]{%
  \href{http://www.ams.org/mathscinet-getitem?mr=#1}{#2}
}
\providecommand{\href}[2]{#2}
\begin{thebibliography}{Eno07}

\bibitem[CM95]{CMbook}
Carl~C. Cowen and Barbara~D. MacCluer, \emph{Composition operators on spaces of
  analytic functions}, CRC Press, 1995.

\bibitem[CM03]{CM03}
\bysame, \emph{Schroeder's equation in several variables}, Taiwanese Journal of
  Mathematics (2003).

\bibitem[Eno07]{Enoch}
R.~Enoch, \emph{Formal power series solutions of {S}chroeder's equation},
  Aequationes {M}athematicae (2007).

\bibitem[Koe84]{Koen}
G.~Koenigs, \emph{Recherches sur les int\'{e}grales de certaines \'{e}quations
  fonctionnelles}, Annales scientifiques de l’ \'{E}cole Norm. Sup. (1884).

\end{thebibliography}

\end{document}